\documentclass{amsart}
\usepackage{amssymb, enumitem}
\usepackage[all]{xy}
\usepackage{hyperref, aliascnt}
\usepackage{mathtools}
\usepackage{bbm}
\usepackage{mathrsfs}
\def\today{\number\day\space\ifcase\month\or   January\or February\or
   March\or April\or May\or June\or   July\or August\or September\or
   October\or November\or December\fi\   \number\year}

\newtheorem{lma}{Lemma}[section]

\newaliascnt{thmCt}{lma}
\newtheorem{thm}[thmCt]{Theorem}
\aliascntresetthe{thmCt}

\newaliascnt{corCt}{lma}
\newtheorem{cor}[corCt]{Corollary}
\aliascntresetthe{corCt}

\newaliascnt{propCt}{lma}
\newtheorem{prop}[propCt]{Proposition}
\aliascntresetthe{propCt}

\newtheorem*{thm*}{Theorem}
\newtheorem*{cor*}{Corollary}
\newtheorem*{prop*}{Proposition}

\theoremstyle{definition}

\newaliascnt{pgrCt}{lma}

\aliascntresetthe{pgrCt}

\newaliascnt{dfCt}{lma}
\newtheorem{df}[dfCt]{Definition}
\aliascntresetthe{dfCt}

\newaliascnt{remCt}{lma}
\newtheorem{rem}[remCt]{Remark}
\aliascntresetthe{remCt}

\newaliascnt{remsCt}{lma}

\aliascntresetthe{remsCt}

\newaliascnt{egCt}{lma}
\newtheorem{eg}[egCt]{Example}
\aliascntresetthe{egCt}

\newaliascnt{qstCt}{lma}
\newtheorem{qst}[qstCt]{Question}
\aliascntresetthe{qstCt}

\newaliascnt{pbmCt}{lma}

\aliascntresetthe{pbmCt}

\newaliascnt{notaCt}{lma}

\aliascntresetthe{notaCt}

\newcommand{\beq}{\begin{equation}}
\newcommand{\eeq}{\end{equation}}
\newcommand{\beqa}{\begin{eqnarray*}}
\newcommand{\eeqa}{\end{eqnarray*}}
\newcommand{\bal}{\begin{align*}}
\newcommand{\eal}{\end{align*}}
\newcommand{\bi}{\begin{itemize}}
\newcommand{\ei}{\end{itemize}}
\newcommand{\be}{\begin{enumerate}}
\newcommand{\ee}{\end{enumerate}}

\newcommand{\ep}{\varepsilon}

\newcommand{\Z}{{\mathbb{Z}}}
\newcommand{\R}{{\mathbb{R}}}
\newcommand{\C}{{\mathbb{C}}}
\newcommand{\N}{{\mathbb{N}}}

\newcommand{\Hi}{{\mathcal{H}}}

\newcommand{\B}{{\mathcal{B}}}
\newcommand{\U}{{\mathcal{U}}}

\newcommand{\id}{{\mathrm{id}}}
\newcommand{\ev}{{\mathrm{ev}}}

\newcommand{\spec}{{\mathrm{sp}}}

\newcommand{\spn}{{\mathrm{span}}}

\newcommand{\Aut}{{\mathrm{Aut}}}
\newcommand{\Max}{{\mathrm{Max}}}

\DeclareMathOperator{\Rep}{Rep}

\newcommand{\ca}{$C^*$-algebra}

\newcommand{\lcg}{locally compact group}

\newcommand{\I}{\infty}

\title{Representations of $p$-convolution algebras on $L^q$-spaces}
\date{\today}

\author[Eusebio Gardella]{Eusebio Gardella}
\address{Eusebio Gardella
Mathematisches Institut, Fachbereich Mathematik und Informatik der
Universit\"at M\"unster, Einsteinstrasse 62, 48149 M\"unster, Germany.}
\email{gardella@uni-muenster.de}
\urladdr{www.math.uni-muenster.de/u/gardella/}
\author{Hannes Thiel}
\address{Hannes Thiel
Mathematisches Institut, Fachbereich Mathematik und Informatik der
Universit\"at M\"unster, Einsteinstrasse 62, 48149 M\"unster, Germany.}
\email{hannes.thiel@uni-muenster.de}
\urladdr{www.math.uni-muenster.de/u/hannes.thiel/}
\thanks{The first named author was partially supported by the D.~K. Harrison Prize from the
University of Oregon and by a Postdoctoral Research Fellowship from the Humboldt Foundation.
The second named author was partially supported by the Deutsche Forschungsgemeinschaft (SFB 878).
Part of this work was completed while the authors were taking part in the
Research Program \emph{Classification of operator algebras, complexity,
rigidity and dynamics}, held at the Institut Mittag-Leffler, between
January and April of 2016. We would like to thank the staff and organizers, and
S{\o}ren Eilers in particular, for the hospitality and financial support.}

\begin{document}

\begin{abstract}

For a nontrivial locally compact group $G$,
consider the Banach algebras of $p$-pseudofunctions, $p$-pseu\-do\-mea\-sures,
$p$-convolvers, and the full group $L^p$-operator algebra.
We show that these Banach algebras are operator algebras if and only if $p=2$.
More generally, we show that for $q\in [1,\I)$, these Banach algebras can be represented on an $L^q$-space
if and only if one of the following holds: (a) $p=2$ and $G$ is abelian; or (b) $\left|\frac 1p - \frac 12\right|=\left|\frac 1q - \frac 12\right|$.
This result can be interpreted as follows: for $p,q\in [1,\I)$, the $L^p$- and $L^q$-representation theories of a group are incomparable, except
in the trivial cases when they are equivalent.

As an application, we show that, for distinct $p,q\in [1,\I)$, if the $L^p$ and $L^q$ crossed products of a topological dynamical system are
isomorphic, then $\frac 1p + \frac 1q=1$. In order to prove this, we study the following relevant aspects of $L^p$-crossed products: existence of approximate
identities, duality with respect to $p$, and existence of canonical isometric maps from group algebras into their multiplier algebras.
\end{abstract} 
\maketitle
\tableofcontents

\section{Introduction}

We say that a Banach algebra is an operator algebra if it admits an isometric representation as bounded
operators on a Hilbert space.
Associated to any locally compact group $G$ there are three fundamentally important operator algebras:
its reduced group $C^*$-algebra $C^*_\lambda(G)$, its full group $C^*$-algebra $C^*(G)$, and its group
von Neumann algebra $L(G)$. These are, respectively, the Banach algebra generated by the left regular
representation of $G$ on $L^2(G)$; the universal \ca\ with respect to unitary representations of
$G$ on Hilbert spaces; and the weak$^*$ closure (also called ultraweak closure) of
$C^*_\lambda(G)$ in $\B(L^2(G))$. (We canonically identify $\B(L^2(G))$ with the dual of the projective
tensor product $L^2(G)\widehat{\otimes}L^2(G)$.)
Equivalently, $L(G)$ is the double commutant of $C^*_\lambda(G)$ in $\B(L^2(G))$.

These operator algebras admit generalizations to representations of $G$ on $L^p$-spaces, for $p\in [1,\I)$.
The analog
of $C^*_\lambda(G)$ is the algebra $F^p_\lambda(G)$ of $p$-pseu\-do\-func\-tions on $G$, introduced
by Herz in \cite{Her_SynSbgps} and originally denoted by $PF_p(G)$. The analog of $C^*(G)$ is the full group $L^p$-operator
algebra $F^p(G)$, defined by Phillips in \cite{Phi_CPLp}. Finally, the von Neumann algebra $L(G)$ has two
analogs, at least for $p\neq 1$: the algebra $PM_p(G)$ of $p$-pseudomeasures, which is the weak$^*$ closure of $F^p_\lambda(G)$ in
$\B(L^p(G))$ (where we canonically identify $\B(L^p(G))$ with the dual of the projective tensor product
$L^p(G)\widehat{\otimes}L^p(G)^*$);
and the algebra $CV_p(G)$ of $p$-convolvers, which is the double commutant of $F^p_\lambda(G)$
in $\B(L^p(G))$ (it is also the commutant of the right regular representation).
Both $PM_p(G)$ and $CV_p(G)$ were introduced in \cite{Her_SynSbgps}.

These objects, and related ones, have been studied by a number of authors in the last three decades.
For instance, see \cite{Cow_ApprProp}, \cite{NeuRun_column}, \cite{Run_QSLp}, \cite{DawSpr_ConvPMeas},
and the more recent papers \cite{Phi_CPLp}, \cite{Phi_MultDomain}, \cite{GarThi_GpsLp}, and
\cite{GarThi_GpsLpFunct}. In \cite{GarThi_Quot}, it is shown that there is a certain quotient of $F^p(\Z)$
which cannot be represented on an $L^p$-space (in fact, on any $L^q$-space for $q\in [1,\I)$), thus answering
a 20-year-old question of Le Merdy.

Despite the advances in the area, some basic questions remain open.
One important
open problem is whether $PM_p(G)=CV_p(G)$ for all $p\in (1,\I)$ and for all locally compact
groups $G$. This is known to be true when $p=2$ (both algebras agree with $L(G)$), essentially by the double commutant theorem.
It is immediate that $PM_p(G)\subseteq CV_p(G)$ in general, while Herz showed in \cite{Her_SynSbgps}
that equality holds for all $p$ if $G$ is amenable, a
result that was later generalized by Cowling in \cite{Cow_ApprProp} to groups with the approximation property.

A less studied problem is the following. By universality of $F^p(G)$,
there is a canonical contractive homomorphism $\kappa_p\colon F^p(G)\to F^p_\lambda(G)$ with dense range. For $p=2$, this
map is known to be a quotient map, and for $p=1$ it is an isomorphism regardless of $G$. On the other hand,
we do not know if $\kappa_p$ is also a quotient map for all other values of $p$. In fact,
we do not even know whether $\kappa_p$ is surjective. If this map is not necessarily surjective, can it be
injective without the group being amenable? (By Theorem~3.7 in~\cite{GarThi_GpsLp}, $G$ is amenable if and only
if $\kappa_p$ is bijective for some (equivalently, for all) $p\in (1,\I)$. This result was independently obtained by
Phillips in \cite{Phi_CPLp} and \cite{Phi_MultDomain}, using different methods.) In this case, it would be interesting
to describe precisely for what groups (and H\"older exponents) the map $\kappa_p$ is injective but not surjective.

Questions of the nature described above would in principle be easier to tackle if the objects considered had a
better understood structure, as is the case for operator algebras. Despite the fact
that the Banach algebras $F^p(G), F^p_\lambda(G), PM_p(G)$ and $CV_p(G)$ have natural representations as operators
on an $L^p$-space, this by itself does not rule out having isometric representations on Hilbert spaces as well;
see, for example, \cite{BleLeM_quotients}. It is therefore not a priori clear whether the
$L^p$-analogs of group operator algebras can be isometrically represented on Hilbert spaces.

In this paper, we settle this question negatively. Indeed, we show in \autoref{thm:FplambdaGLq} that for a nontrivial
locally compact group $G$, and for $p\in [1,\I)\setminus\{2\}$, none of the algebras $F^p(G), F^p_\lambda(G),
PM_p(G)$, or $CV_p(G)$ can be isometrically represented on a Hilbert space. This result generalizes Theorem~2.2 in~\cite{NeuRun_column},
where Neufang and Runde assume that $G$ is amenable and has
a closed infinite abelian subgroup. More generally, for $p,q\in [1,\I)$ we show that
the algebras $F^p_\lambda(G)$, $F^p(G)$, $PM_p(G)$, or $CV_p(G)$ can be isometrically represented on
an $L^q$-space if and only if one of the following holds:
\be\item $p=2$ and $G$ is abelian; or
\item $\left|\frac{1}{p}-\frac 12\right|=\left|\frac{1}{q}-\frac 12\right|$. (This is equivalent to either $p=q$ or $\frac{1}{p}+\frac 1q=1$.) \ee

This result can be interpreted as
asserting that the $L^p$- and $L^q$-representation theories of a nontrivial group are incomparable, whenever they are not ``obviously'' equivalent.
As a consequence, it follows that if there is an isometric Banach algebra isomorphism
$F^p_\lambda(G)\cong F^q_\lambda(G)$ (or between full group algebras, pseudomeasures
or convolvers) for distinct $p,q\in [1,\I)$, then $\frac{1}{p}+\frac 1q=1$; see \autoref{cor:FpGFqG}. The converse also holds; see \autoref{thm:duality}. 

As an application, we show that, for $p,q\in [1,\I)$, if the $L^p$- and $L^q$-crossed products of a topological dynamical system are
isometrically isomorphic, then $\left|\frac 1p - \frac 12\right|=\left|\frac 1q - \frac 12\right|$; see \autoref{cor:CrossProdsIsom}.
Since we do not know whether
an isomorphism $F^p(G,X,\alpha)\to F^q(G,X,\alpha)$ (or $F^p_\lambda(G,X,\alpha)\to F^q_\lambda(G,X,\alpha)$)
must necessarily respect the group action $\alpha$, \autoref{cor:FpGFqG} is
not enough to obtain the conclusion. This means that even if we are only interested in isomorphisms of crossed products, we are
forced to consider arbitrary representations of $p$-convolution algebras on $L^q$-spaces. In order to obtain these results,
we need to develop the theory of $L^p$-crossed products further, and we do so by exploring the following fundamental aspects: existence of approximate
identities, duality with respect to $p$, and existence of canonical isometric maps from group algebras into their multiplier algebras.

We give an outline of the proof of our main result (\autoref{thm:FplambdaGLq}).
Let $E$ be an $L^q$-space and let $\varphi\colon F^p_\lambda(G)\to \B(E)$ be
an isometric representation. (Similar arguments apply for the algebras $F^p(G)$, $PM_p(G)$, or $CV_p(G)$.)
\be
\item[Step 1.] The case $q=2$ is treated separately, since in this case one can show that $F^p_\lambda(G)$ is a \ca; see \autoref{thm:FplambdaG}.
So assume that $q\neq 2$.
\item[Step 2.] Using results in \cite{GarThi_ExtBid}, we may assume that $\varphi$ is non-degenerate.
\item[Step 3.] Functoriality properties of $F^p_\lambda$ with respect to subgroups (\cite{GarThi_GpsLpFunct})
allow us to further reduce the problem to the case where $G$ is a cyclic group (finite or infinite); see \autoref{lma:SubgpH}.
\item[Step 4.] For cyclic $G$, we need to know that $\varphi(F^p_\lambda(G))$ is isometrically isomorphic to $F^q_\lambda(G)$. This
requires non-trivial results from \cite{GarThi_BanAlg} on spectral configurations. (For instance, we use the fact that $F^q_\lambda(\Z)$
is the unique $L^q$-operator algebra generated by an invertible isometry whose Gelfand transform is not surjective.) We conclude
that there is an isometric isomorphism $F^p_\lambda(G)\cong F^q_\lambda(G)$.
\item[Step 5.] When $G=\Z$, we show that $\left|\frac{1}{p}-\frac 12\right|=\left|\frac{1}{q}-\frac 12\right|$ in \autoref{thm: FpZ not isom}
using elementary computations.
\item[Step 6.] When $G=\Z_n$, we use the existence, for $1\leq p \leq q\leq 2$, of a canonical, contractive map
$\gamma_{p,q}\colon F^p(G)\to F^q(G)$ with dense range (see \autoref{thm:gammapq}), to show that the Gelfand
transform of $F^p_\lambda(G)$ is an isometric isomorphism, thus reducing the problem to the case in Step~1 above.
\ee



We wish to point out that the use of spectral configurations can be avoided if one is only interested in the case
$\left|\frac{1}{p}-\frac 12\right|>\left|\frac{1}{q}-\frac 12\right|$. Indeed, this situation
can be entirely dealt with the maps $\gamma_{p,q}$ (see \autoref{prop:FpGReprLq} or \autoref{thm:FplambdaGQSLq} for a
similar argument, but in a different context). Also, the fact that we can assume representations to be non-degenerate
(Step 2) is by no means obvious, and the paper \cite{GarThi_ExtBid} grew out of our attempts to prove this.

\vspace{0.3cm}

\subsection{Notation}\label{subs:Not}
We take $\N=\{1,2,\ldots\}$. For $n\in \N$ and $p\in [1,\I)$,
we write $\ell^p_n$ in place of $\ell^p(\{1,\ldots,n\})$, and we write $\ell^p$ in place of $\ell^p(\Z)$.
For a Banach space $E$, we write $\B(E)$ for the Banach algebra of bounded linear operators on $E$.
For $p\in (1,\I)$, we denote by $p'$ its conjugate (H\"older) exponent, which is determined by the identity $\frac{1}{p}+\frac{1}{p'}=1$.
Consistently, for a Banach space $E$, we denote its dual space by $E'$, and for a linear map $\pi\colon E\to F$ between Banach spaces $E$ and
$F$, we denote by $\pi'\colon F'\to E'$ its transpose map.

Locally compact groups are assumed to be Hausdorff, and will always be implicitly endowed with a (fixed) left Haar measure, which will be chosen
to be the counting measure whenever the group is discrete. The left Haar measure of a locally compact group $G$ will be denoted by $\mu$, and we
will denote by $\nu$ the right Haar measure on $G$ determined by $\nu(U)=\mu(U^{-1})$ for all measurable sets $U\subseteq G$. The modular function 
of $G$ will be denoted by $\Delta\colon G\to \R_+$. We will repeatedly use the following identities:
\[\int_G f(t)\ d\mu(t)=\int_G\Delta(t^{-1})f(t^{-1})\ d\mu(t)=\int_G f(t^{-1})\ d\nu(t)=\int_G \Delta(t) f(t) \ d\nu(t),\]
valid for all $f\in L^1(G,\mu)$.

\section{Preliminaries}

In this section, we recall and collect the necessary definitions and theorems that will be used throughout the paper.
The only thing in this section that is really new is \autoref{prop:defagree}.

We begin by defining the main objects of study of this work.

\begin{df}\label{df:FpG}
Let $G$ be a locally compact group, and let $p\in [1,\I)$. Denote by $\Rep_p(G)$ the class
of all contractive representations of $L^1(G)$ on $L^p$-spaces. The \emph{full group $L^p$-operator
algebra of $G$}, denoted $F^p(G)$, is the completion of $L^1(G)$ in the norm given by
\[\|f\|_{F^p(G)}=\sup\left\{\|\pi(f)\|\colon \pi\in \Rep_p(G)\right\}\]
for $f\in L^1(G)$.

Denote by $\lambda_p\colon L^1(G)\to \B(L^p(G))$ the left regular representation, which is given by
$\lambda_p(f)\xi=f\ast\xi$
for $f\in L^1(G)$ and $\xi\in L^p(G)$. The \emph{algebra of $p$-pseudofunctions on $G$} (sometimes also called
\emph{reduced group $L^p$-operator algebra of $G$}), here denoted $F^p_\lambda(G)$, is the completion
of $L^1(G)$ in the norm
\[\|f\|_{F^p_\lambda(G)}=\|\lambda_p(f)\|_{\B(L^p(G))}\]
for $f\in L^1(G)$.
\end{df}

The algebra $F^p(G)$ has been defined in \cite{Phi_CPLp} and \cite{GarThi_GpsLp} as the completion
of $L^1(G)$ with respect to \emph{non-degenerate}, contractive representations on $L^p$-spaces. In the
proposition below, we show that this distinction is irrelevant. We recall the following, which is a
particular case of a result from \cite{GarThi_ExtBid}.

\begin{thm}\label{thm:ProjEssSubsp} (\cite{GarThi_ExtBid}).
Let $A$ be a Banach algebra with a left contractive approximate identity,
let $E$ be a reflexive Banach space, and let $\varphi\colon A\to \B(E)$ be a contractive homomorphism.
Denote by $E_0$ the essential subspace of $\varphi$, this is, $E_0=\overline{\spn} \ \varphi(A)E$.
Then there exists a contractive idempotent $e\in\B(E)$ satisfying $e(E)=E_0$. \end{thm}

\begin{prop}\label{prop:defagree}
Let $p\in [1,\I)$ and let $G$ be a locally compact group. If $E$ is an $L^p$-space and
$\pi\colon L^1(G)\to \B(E)$ is a contractive representation, then there exist an $L^p$-space $F$
and a contractive, non-degenerate representation $\varphi\colon L^1(G)\to \B(F)$ such that
\[\|\pi(f)\|\leq \|\varphi(f)\|\]
for all $f\in L^1(G)$. In particular, \autoref{df:FpG} agrees with the definitions given in
\cite{Phi_CPLp} and \cite{GarThi_GpsLp}.
\end{prop}
\begin{proof}
We treat the case $p=1$ first. In this case, we may take $\varphi$ to be the left regular
representation $\lambda_1$. Indeed, let $\pi\colon L^1(G)\to \B(E)$ be a contractive representation
on an $L^1$-space, and let $(a_j)_{j\in J}$ be a contractive
approximate identity for $L^1(G)$. Then
\[\|\pi(f)\|\leq \|f\|_1=\lim_{j\in J} \|f\ast a_j\|_1=\lim_{j\in J}\|\lambda_1(f)a_j\|_1\leq \|\lambda_1(f)\|,\]
for all $f\in L^1(G)$, as desired.

Assume now that $p>1$. Let $\pi\colon L^1(G)\to \B(E)$ be a contractive representation
on an $L^p$-space. Since $E$ is reflexive and $L^1(G)$ has a contractive approximate identity, it follows
from \autoref{thm:ProjEssSubsp} that there exists a contractive projection $e\in\B(E)$ such that $e(E)$ is the essential subspace $E_0$ of $\pi$.
Then $E_0$ is an $L^p$-space by Theorem~6 in~\cite{Tzafriri}. Let $\varphi\colon L^1(G)\to \B(E_0)$ be the restriction
of $\pi$. Then $\varphi$ is a non-degenerate, contractive representation, and it is clear that $\|\pi(f)\|=\|\varphi(f)\|$
for all $f\in L^1(G)$.
\end{proof}

The following duality principle was established in Proposition~2.18 of~\cite{GarThi_GpsLp}.

\begin{prop}\label{prop:dualityFullRed}
Let $G$ be a \lcg, let $p\in (1,\infty)$, and let $p'$ be its conjugate exponent.
Then the inversion map on $G$ extends to a canonical isometric isomorphism
$F^p(G)\cong F^{p'}(G)$.
\end{prop}

The next result is a combination of Theorem~2.30 and Corollary~3.20 in~\cite{GarThi_GpsLp}. The
case $2\leq q\leq p<\infty$ is obtained using duality, so we omit it.

\begin{thm}\label{thm:gammapq}
Let $G$ be a \lcg.
If $1\leq p \leq q\leq 2$, then the identity map on $L^1(G)$ extends to a contractive homomorphism
\[
\gamma_{p,q}\colon F^p(G)\to F^q(G)
\]
with dense range. In particular, $\|f\|_{F^q(G)}\leq \|f\|_{F^p(G)}$ for every $f\in L^1(G)$.

If, moreover, $G$ is amenable, then:
\begin{enumerate}
\item[(a)] $\gamma_{p,q}$ is injective.
\item[(b)] $\gamma_{p,q}$ is surjective if and only if $G$ is finite.
\end{enumerate}
\end{thm}

\begin{rem} \label{rem:p2}
When $p=2$, we have $F^p(G)=C^*(G)$, the full group $C^*$-algebra, and $F^p_\lambda(G)=C^*_\lambda(G)$,
the reduced group $C^*$-algebra. In particular, when $G$ is abelian, then $\gamma_{p,2}\colon F^p(G)\to C^*(G)\cong C_0(\widehat{G})$
is the Gelfand transform.
\end{rem}

It follows from
universality of the norm of $F^p(G)$ that the identity map on $L^1(G)$ extends to a contractive homomorphism
$\kappa_p\colon F^p(G)\to F^p_\lambda(G)$ with dense range.
The following is part of Theorem~3.7 in \cite{GarThi_GpsLp}. This result has been independently obtained by
Phillips, and will appear in \cite{Phi_MultDomain}.

\begin{thm}\label{thm:Amen}
Let $G$ be a locally compact group, and let $p\in(1,\infty)$.
The following are equivalent:
\begin{enumerate}
\item
The group $G$ is amenable.
\item
The canonical map $\kappa_p\colon F^p(G)\to F^p_\lambda(G)$ is an (isometric) isomorphism.
\end{enumerate}
\end{thm}

In view of the above result, we will identify $F^p(G)$ and $F^p_\lambda(G)$ in a canonical manner whenever $G$ is
amenable.

\section{\texorpdfstring{$F^p(\Z)$}{Fp(Z)} is not isomorphic to \texorpdfstring{$F^q(\Z)$}{Fq(Z)}}

The main step in the proof of \autoref{thm:FplambdaGLq} is to use the functoriality properties studied in \cite{GarThi_GpsLpFunct} to reduce
the statement to the case where $G$ is a cyclic group (finite or infinite). While the case of a finite cyclic group will be dealt with
using spectral configurations introduced in~\cite{GarThi_BanAlg}, the case of the infinite cyclic group has to be dealt with separately,
and we do so in this section. Our goal here is to prove the following: for $p,q\in [1,\I)$,
there is an isometric isomorphism $F^p(\Z)\cong F^q(\Z)$ as Banach algebras if and only if
$\left|\frac 1p-\frac 12\right|=\left|\frac 1q-\frac 12\right|$ (which is equivalent to $p$ and $q$ being
either equal or H\"older conjugate); see \autoref{thm: FpZ not isom}. In order to prove this, we will
reduce to the problem of ruling out the possibility of a \emph{canonical} isomorphism for the cyclic group $\Z_2$.
In this case, the norm of a certain element can be explicitly computed, yielding the result; see \autoref{prop: FpZ2 not isom}.

It should be pointed out that the methods of this section do not seem to generalize to other finite cyclic groups besides $\Z_2$.
We will postpone dealing with such groups until the next section; see \autoref{thm:FplambdaGLq}, and specifically the proof of
Claim 2 in it. The approach we take here for the case of $\Z_2$ has the advantage of being completely elementary and it does not
depend on results from other works (in particular, we do not need here anything about spectral configurations from \cite{GarThi_BanAlg}).

We begin by looking at the group $L^p$-operator algebra of a finite cyclic group, which by amenability coincides with the algebra
of $p$-pseudofunctions (this is the presentation we actually use). 

\begin{eg} \label{eg: FpZn}
Let $n\in \N$ and let $p\in [1,\I)$. Then $L^p$-group algebra $F^p(\Z_n)$ is the Banach
subalgebra of $\B(\ell^p_n)$ generated by the cyclic shift $s_n$ of order $n$
\[s_n= \left( \begin{array}{ccccc}
0 &  &  &  & 1 \\
1 & 0 & &  &  \\
& \ddots & \ddots &  &  \\
&  & \ddots & 0 &  \\
&  &  & 1 & 0 \end{array} \right).\]
(The algebra $\B(\ell^p_n)$ is just $M_n$ with the $L^p$-operator norm.) It is easy to check that $F^p(\Z_n)$ is algebraically isomorphic to $\C^n$,
but the canonical embedding $F^p(\Z_n)\to M_n$ is not as diagonal matrices.
In fact, computing the norm of an element in $F^p(\Z_n)$ is challenging for $p$ other than
from 1 and 2, essentially because computing $p$-norms of matrices that are not diagonal is difficult.
Indeed, let
$\omega_n=e^{\frac{2\pi i}{n}}$, and set
\[
u_n= \frac{1}{\sqrt{n}} \left( \begin{array}{ccccc}
1 & 1 & 1 & \cdots & 1 \\
1 &\omega_n &\omega_n^2 & \cdots &\omega_n^{n-1} \\
1 &\omega_n^2 &\omega_n^4 & \cdots &\omega_n^{2(n-1)} \\
\vdots & \vdots & \vdots & \ddots & \vdots \\
1 &\omega_n^{n-1} &\omega_n^{2(n-1)} & \cdots &\omega_n^{(n-1)^2} \end{array} \right).
\]

If $\xi=(\xi_0,\ldots,\xi_{n-1})\in \C^n\cong F^p(\Z_n)$, then its norm in $F^p(\Z_n)$ is
\[
\|\xi\|_{F^p(\Z_n)}=\left\| u_n \left( \begin{array}{ccccc}
\xi_0 &  &  & \\
 & \xi_1 & & \\
 &  & \ddots & \\
&  &  &  \xi_{n-1}  \end{array} \right)u_n^{-1}\right\|_{\B(\ell^p_n)}.
\]

The matrix $u_n$ is unitary (its inverse is the transpose of its conjugate), and hence $\|\xi\|_{F^2(\Z_n)}=\|\xi\|_\infty$. 
Moreover, if $1\leq p\leq q\leq 2$,
then $\|\cdot\|_{F^q(\Z_n)}\leq \|\cdot\|_{F^{p}(\Z_n)}$ by \autoref{thm:gammapq}.
In particular, the norm $\|\cdot\|_{F^p(\Z_n)}$ always dominates the norm
$\|\cdot\|_\I$.\end{eg}

The automorphism group of $F^p(\Z_n)$ is not easy to describe when $p\neq 2$, since not every permutation of the
coordinates of $\C^n$ is isometric. The next proposition asserts that the cyclic shift on $\C^n$ is isometric.

\begin{prop} \label{prop: shift invariance norm on FpZn}
Let $n\in \N$ and let $p\in [1,\I)$. Denote by $\tau\colon\C^n \to \C^n$ the cyclic forward shift, this is,
\[
\tau(\xi_0,\ldots,\xi_{n-1})=(\xi_{n-1},\xi_0,\ldots,\xi_{n-2})
\]
for all $(\xi_0,\ldots,\xi_{n-1})\in \C^n$. Then $\tau$ induces an isometric isomorphism $F^p(\Z_n)\to F^p(\Z_n)$.
\end{prop}
\begin{proof}
We follow the notation from \autoref{eg: FpZn}, except that we write $u$
in place of $u_n$, and we write $s$ in place of $s_n$. 

For $\xi\in \C^n$, let $d(\xi)$ denote the diagonal $n\times n$ matrix with $d(\xi)_{j,j}=\xi_j$ for $j=0,\ldots, n-1$.
Denote by $\rho\colon\C^n\to M_n$ the unital homomorphism given by $\rho(\xi)= ud(\xi)u^{-1}$ for $\xi\in\C^n$.
Then
\[
\|\xi\|_{F^p(\Z_n)} =\|\rho(\xi)\|_{\B(\ell^p_n)} = \|ud(\xi)u^{-1}\|_{\B(\ell^p_n)}
\]
for all $\xi\in \C^n$.\\
\indent Set $\omega=(1,\omega_n^1,\ldots,\omega_n^{n-1})\in\C^n$, and denote by $\overline{\omega}$ its (coordinate-wise) conjugate.
Given $\xi\in\C^n$, one checks that
\[
\mathrm{d}(\tau(\xi)) = s d(\xi) s^{-1}, \ \  us = d(\omega)u, \ \ \text{ and }\ \ s^{-1}u^{-1} = u^{-1} d(\overline{\omega}).
\]
It follows that
\begin{align*}
\|\tau(\xi)\|_{F^p(\Z_n)} &=\|u \mathrm{d}(\tau(\xi))u^{-1}\|_{\B(\ell^p_n)} \\
&=\|u s d(\xi) s^{-1} u^{-1}\|_{\B(\ell^p_n)} \\
&=\|d(\omega) u d(\xi) u^{-1} d(\overline{\omega})\|_{\B(\ell^p_n)}.
\end{align*}
Since $d(\omega)$ and $d(\overline{\omega})$ are isometries of $\ell^p_n$, we conclude that
\[
\|\tau(\xi)\|_{F^p(\Z_n)}
=\|d(\omega) u \tau(\xi) u^{-1} d(\overline{\omega})\|_{\B(\ell^p_n)}
=\|u d(\xi) u^{-1}\|_{\B(\ell^p_n)}
=\|\xi\|_{F^p(\Z_n)},
\]
as desired. \end{proof}

The fact that $F^p(\Z_2)$ is isometrically isomorphic to $F^q(\Z_2)$ only in the ``obvious'' cases can be proved directly
by computing the norm of a special element, as we show below.

\begin{prop}\label{prop: FpZ2 not isom}
Let $p, q\in [1,\infty)$. Then $F^p(\Z_2)$ is isometrically isomorphic, as Banach algebras, to $F^{q}(\Z_2)$ if and only if
$\left|\frac{1}{p}-\frac 12\right|=\left|\frac{1}{q}-\frac 12\right|$.
\end{prop}
\begin{proof}
The ``if'' implication follows from \autoref{prop:dualityFullRed}.
We proceed to show the ``only if'' implication.
Given $t\in [1,\infty)$, we claim that
\[\|(1,i)\|_{F^t(\Z_2)}=2^{\left|\frac{1}{t}-\frac{1}{2}\right|}.\]

By \autoref{prop:dualityFullRed}, the quantity on the left-hand side remains unchanged if one replaces $t$
with its conjugate exponent. Since the same holds for the quantity on the right-hand side, it follows that it is
enough to prove the claim for $t\in [1,2]$.\\
\indent  Define a continuous function $\delta\colon[1,2]\to \mathbb{R}$ by $\delta(t)=\|(1,i)\|_{F^t(\Z_2)}$ for $t\in [1,2]$.
Let $a$ be the matrix
\[
a=\frac{1}{2} \begin{pmatrix}
1+i & 1-i \\
1-i & 1+i
\end{pmatrix}.
\]
Then $\delta(t)= \left\| a \right\|_{\B(\ell^t_2)}$ for all $t\in [1,2]$.
The values of $\delta$ at $t=1$ and $t=2$ are easy to compute, and we have $\delta(1)=2^{\frac{1}{2}}$ and $\delta(2)=1$.
Fix $t\in (1,2)$ and let $\theta\in (0,1)$ satisfy
\[\frac{1}{t}=\frac{1-\theta}{1} + \frac{\theta}{2}.\]
Using the Riesz-Thorin Interpolation Theorem in $[1,2]$, we conclude that
\[
\delta(t) \leq\delta(1)^{1-\theta}\cdot\delta(2)^{\theta}=2^{\frac{1}{2}(\frac{2}{t}-1)}\cdot 1= 2^{\frac{1}{t}-\frac{1}{2}}.\]

For the converse inequality, fix $t\in [1,2]$ and consider the vector 
$\xi=\left(\begin{smallmatrix}
1 \\ 
0 
\end{smallmatrix}\right) \in \ell^t_2$. Then $\|\xi\|_t=1$ and
$ax = \frac{1}{2}(\begin{smallmatrix}
1+i \\ 
1-i
\end{smallmatrix})$.
We have
\[
\left\|\frac{1}{2}\begin{pmatrix}
1+i \\ 1-i
\end{pmatrix}\right\|_t
= \frac{1}{2} ( |1+i|^t + |1-i|^t)^{\frac{1}{t}}
= 2^{(\frac{1}{t}-\frac{1}{2})}.
\]
We conclude that
\[
\delta(t)
=\left\|a\right\|_{\B(\ell^t_2)}
\geq \frac{\left\|ax \right\|_t}{\|\xi\|_t}
= 2^{(\frac{1}{t}-\frac{1}{2})}.
\]
This shows that $\delta(t)=2^{(\frac{1}{t}-\frac{1}{2})}$ for all $t\in[1,2]$, and the claim is proved.

Now let $p,q\in [1,\infty)$ and let $\psi\colon F^p(\Z_2)\to F^{q}(\Z_2)$ be an isometric isomorphism.
Since $\psi$ is an algebra isomorphism, we must have either $\psi(x,y)=(x,y)$ or $\psi(x,y)=(y,x)$ for all $(x,y)\in\C^2$.
By \autoref{prop: shift invariance norm on FpZn}, the flip $(x,y)\mapsto (y,x)$ is an isometric isomorphism of
$F^{q}(\Z_2)$, so we may assume that $\psi$ is the identity map on $\C^2$. It follows that $\|(1,i)\|_{F^p(\Z_2)}
= \|(1,i)\|_{F^{q}(\Z_2)}$, so $\left|\frac{1}{p}-\frac{1}{2}\right| = \left|\frac{1}{q}-\frac{1}{2}\right|$ and the proof is complete.
\end{proof}

\begin{rem}
Adopt the notation from the proof above. Then the function $\delta$ attains the upper bound given by the Riesz-Thorin
Interpolation Theorem, which is a rare situation. This fortunate coincidence makes the argument possible, but it is not
clear to us how to generalize these computations to other cyclic groups, or even to $\Z_3$. However, knowing the result
just for $\Z_2$ is enough to prove \autoref{thm: FpZ not isom}.
\end{rem}

The following is probably standard, but we have not been able to find a reference in the literature. Accordingly,
we prove it here.

\begin{prop}\label{prop:homeoS1}
Let $f\colon S^1\to S^1$ be a homeomorphism. Then there exists $\zeta\in S^1$ such that
$f(-\zeta)=-f(\zeta)$.
\end{prop}
\begin{proof}
Regard $f$ as a homeomorphism $h\colon [0,2\pi]\to [0,2\pi]$ with either $h(0)=0$ (if $f$ is orientation preserving)
or $h(0)=2\pi$ (if $f$ is orientation reversing). Without loss of generality, assume that $h(0)=0$, and hence
that $h$ is strictly increasing with $h(2\pi)=2\pi$. Let $g\colon [0,\pi]\to\R$ be given by
\[g(t)=\frac{h(t+\pi)-h(t)}{\pi}\]
for $t\in [0,\pi]$. We need to find $s\in [0,\pi]$ with $g(s)=1$. If $h(\pi)=\pi$, then $g(\pi)=1$
and we are done. If $h(\pi)>\pi$, we
have $g(0)>1$, while $g(\pi)<1$. Likewise, if $h(\pi)<\pi$, then $g(0)<1$ and $g(\pi)>1$.
In either case, the conclusion then follows from the Mean Value Theorem.
\end{proof}

We are now ready to prove that for $p,q\in [1,\infty)$, the algebras $F^p(\Z)$
and $F^{q}(\Z)$ are (abstractly) isometrically isomorphic only in the trivial case when
$\left|\frac{1}{p}-\frac{1}{2}\right|=\left|\frac{1}{q}-\frac{1}{2}\right|$.
This result should be compared with part~(2) of Corollary~3.20 
in~\cite{GarThi_GpsLp}, here reproduced as \autoref{thm:gammapq},
where only the canonical homomorphism is considered. 
The strategy will be to use \autoref{prop:homeoS1}
to reduce to the case when the group is $\Z_2$, which is \autoref{prop: FpZ2 not isom}.
The fact that the spectrum of $F^p(\Z)$ is the circle is crucial in
our proof, and we do not know how to generalize these methods to directly deal with, for example, $\Z^2$.

\begin{thm}\label{thm: FpZ not isom}
Let $p,q\in [1,\infty)$.
Then $F^p(\Z)$ and $F^q(\Z)$ are isometrically isomorphic (as Banach algebras)
if and only if $\left|\frac 1p-\frac 12\right|=\left|\frac 1q-\frac 12\right|$.
\end{thm}
\begin{proof}
The ``if'' implication follows from \autoref{prop:dualityFullRed}. Let us show the converse.

By Proposition~3.13 in \cite{GarThi_GpsLp}, the maximal ideal spaces of $F^p(\Z)$ and $F^q(\Z)$ are canonically homeomorphic to $S^1$.
We let $\Gamma_p\colon F^p(\Z)\to C(S^1)$ denote the Gelfand transform, which maps the canonical
generating invertible isometry $u\in F^p(\Z)$, associated to $1\in \Z$, to the canonical inclusion $\iota\colon S^1\to \C$.

Let $\varphi\colon F^p(\Z)\to F^q(\Z)$ be an isometric isomorphism.
Then $\varphi$ induces a homeomorphism $f\colon S^1\to S^1$ that maps $z\in S^1$ to the unique point $f(z)\in S^1$
satisfying $\ev_z\circ\varphi = \ev_{f(z)} \colon F^p(\Z)\to\C$. 

Use \autoref{prop:homeoS1} to choose $\zeta_0\in S^1$ such that $f(-\zeta_0)=-f(\zeta_0)$.
Denote by $\pi_p\colon F^p(\Z)\to F^p(\Z_2)$ and $\pi_q\colon F^p(\Z)\to F^q(\Z_2)$ the canonical homomorphisms associated with
the surjective map $\Z\to \Z_2$, which are quotient maps by Theorem~2.5 in~\cite{GarThi_GpsLpFunct}.
Let $\omega_{\zeta_0}\colon F^p(\Z)\to F^p(\Z)$ be the isometric isomorphism induced by multiplying the canonical
generator $u\in F^p(\Z)$ by $\zeta_0$. Analogously, let $\omega_{f(\zeta_0)}\colon F^q(\Z)\to F^q(\Z)$ be the
isometric isomorphism induced by multiplying $u\in F^q(\Z)$ by $f(\zeta_0)$.
Then the following diagram commutes:

\begin{center}
\makebox{
\xymatrix{ & C(S^1) \ar[r]^-{f^*} & C(S^1) & \\
F^p(\Z) \ar[d]_{\pi_p}  & F^p(\Z)\ar[l]_-{\omega_{\zeta_0}} \ar[r]_-{\varphi} \ar[u]^-{\Gamma_p}
& F^q(\Z) \ar[u]_-{\Gamma_q} \ar[r]^-{\omega_{f(\zeta_0)}} & F^q(\Z) \ar[d]^{\pi_q} \\
F^p(\Z_2)\ar@{-->}_-{\widehat{\psi}}[rrr] & & &  F^q(\Z_2).
}}
\end{center}

Define an isometric isomorphism $\psi\colon F^p(\Z)\to F^q(\Z)$ by $\psi=\omega_{f(\zeta_0)}\circ\varphi\circ\omega_{\zeta_0}^{-1}$. 
One checks that $\psi(\ker(\pi_p))=\ker(\pi_q)$.
It follows that $\psi$ induces an isometric isomorphism $\widehat{\psi}\colon F^p(\Z_2)\to F^q(\Z_2)$.
By \autoref{prop: FpZ2 not isom}, this implies that $\left|\frac 1p-\frac 12\right|=\left|\frac 1q-\frac 12\right|$, as desired.
\end{proof}

\section{Nonrepresentability on \texorpdfstring{$L^q$}{Lq}-spaces}

The goal of this section is to show that for a nontrivial locally compact group $G$ and for $p, q\in [1,\infty)$ distinct, the Banach algebras
$F^p(G)$, $F^p_\lambda(G)$, $PM_p(G)$ and $CV_p(G)$ can be represented on an $L^q$-space only in the trivial cases, namely if
either $p=2$ and $G$ is abelian; or if $p$ and $q$ are H\"older conjugate; see \autoref{thm:FplambdaGLq}.
The main step in the proof is to use the functoriality properties of these objects, studied in \cite{GarThi_GpsLpFunct}, to reduce
the statement to the case where $G$ is a cyclic group (finite or infinite). The case of a finite cyclic group will be dealt with
using spectral configurations as in \cite{GarThi_BanAlg}, as well as the canonical maps $\gamma_{p,q}\colon F^p(G)\to F^q(G)$
constructed in Theorem~3.7 in~\cite{GarThi_GpsLp}.

\autoref{prop:FpGReprLq} below is our first preparatory result on representability of full group $L^p$-operator
algebras on $L^q$-spaces. We need some notation first. Let $G$ be a \lcg, and denote by $\Delta\colon G\to \R$
its modular function.
For $f\in L^1(G)$, let $f^\sharp\colon G\to\C$ be given by $f^\sharp(s) = \Delta(s^{-1}) f(s^{-1})$ for all
$s\in G$. It is easy to check that the map $\sharp\colon L^1(G)\to L^1(G)$ is an anti-multiplicative isometric linear map
of order two.

\begin{prop}
\label{prop:FpGReprLq}
Let $G$ be a locally compact group, and let $p,q\in [1,\I)$ satisfy $\left|\frac{1}{p}-\frac 12\right|>\left|\frac{1}{q}-\frac 12\right|$.
Suppose that $F^p(G)$ is isometrically representable on an $L^q$-space.
\be
\item If $p,q\in [1,2]$ or $p,q\in [2,\I)$,
then the identity map on $L^1(G)$ extends to an isometric isomorphism $F^p(G)\cong F^q(G)$.
\item If $p \in [1,2]$ and $q\in [2,\I)$, or if $q \in [1,2]$ and $p\in [2,\I)$, then the map
$\sharp$ on $L^1(G)$ induces, when composed with the transpose map $\B(L^p(G))\to \B(L^{p'}(G))$,
an isometric isomorphism $F^p(G)\cong F^q(G)$.
\ee
\end{prop}
\begin{proof}
(1). The result is trivial when $p=q$. Suppose first that
$1\leq p < q\leq 2$.
Suppose that there exist an $L^q$-space $E$ and an isometric representation $\varphi\colon F^p(G)\to \B(E)$.
Let $\iota_p\colon L^1(G)\to F^p(G)$ be the canonical contractive inclusion with dense range. Then
\[\psi=\varphi\circ\iota_p\colon L^1(G)\to \B(E)\]
is a contractive representation of $L^1(G)$ on an $L^q$-space.

Let $f\in L^1(G)$. Using \autoref{thm:gammapq} at the first step, we deduce that
\[\|f\|_{F^q(G)}\leq \|f\|_{F^p(G)}=\|\psi(f)\|_{\B(F)}\leq \|f\|_{F^q(G)}.\]

Hence $\|f\|_{F^q(G)}=\|f\|_{F^p(G)}$. Since $f\in L^1(G)$ is arbitrary,
it follows that the identity on $L^1(G)$ extends to an isometric isomorphism $F^p(G)\to F^q(G)$,
as desired.

The case $2\leq q< p <\I$ follows from duality; see \autoref{prop:dualityFullRed}.

(2). We can assume, without loss of generality, that $1\leq p\leq 2\leq q<\I$. Denote by $q'\in [1,2]$
the H\"older conjugate exponent of $q$. By \autoref{prop:dualityFullRed}, the map $\sharp\colon L^1(G)\to L^1(G)$,
when composed with the transpose map $\B(L^q(G))\to \B(L^{q'}(G))$, extends to
an isometric isomorphism $F^q(G)\to F^{q'}(G)$. (The details are in the proof of Lemma~2.16
of~\cite{GarThi_GpsLp}, the main point being that given a representation of $G$ on an $L^q$-space,
its dual representation is a representation of $G^{\mathrm{op}}$ on $L^{q'}$ which induces the same norm
on $L^1(G)$. One composes this isometric anti-isomorphism with the map $\sharp$ to get a multiplicative
(isometric) isomorphism.) Since the identity map on $L^1(G)$ extends to an isometric isomorphism $F^p(G)\to F^{q'}(G)$ by
part (1), the result follows.\end{proof}





Besides $F^p(G)$ and $F^p_\lambda(G)$, the other two Banach algebras we will be concerned with, at least
when $p>1$, are the algebra $PM_p(G)$ of $p$-pseudomeasures,
and the algebra $CV_p(G)$ of $p$-convolvers. These are, respectively, the ultraweak closure, and
the bicommutant, of $F^p_\lambda(G)$ in $\B(L^p(G))$.

Algebras of pseudomeasures and of convolvers on groups have been thoroughly studied since their inception by
Herz in the early 70's; see \cite{Her_SynSbgps}.
It is clear that $PM_p(G)\subseteq CV_p(G)$, and it is conjectured that they are equal for every locally
compact group $G$ and every H\"older exponent $p\in (1,\I)$. The conjecture is known to be true if $p=2$,
or if $G$ is amenable (\cite{Her_SynSbgps}), or, more generally, if $G$ has the approximation
property (\cite{Cow_ApprProp}).

Our next goal is showing that these convolution algebras are never operator algebras when $p\neq 2$.
We first need two results about \ca s which are interesting in their own right. The first one is well-known, and
it follows, from example, from Theorem~10 of~\cite{Bon_Minimal}.

\begin{thm}\label{thm:C*incompress}
Let $A$ be a \ca, let $B$ be a Banach algebra, and let $\varphi\colon A\to B$ be a contractive, injective
homomorphism. Then $\varphi$ is isometric.
\end{thm}

\begin{rem}
If in the theorem above $\varphi$ is not assumed to be injective, then the conclusion is that it is
a quotient map. On the
other hand, we must assume that $\varphi$ is contractive, and not merely continuous, for the result to hold;
counterexamples are easy to construct. The result also fails for not necessarily self-adjoint operator
algebras, even for uniform algebras.
\end{rem}

The next fact about $C^*$-algebras is proved in \cite{GarThi_ExtBid}, and had not been noticed before, at least not
in this generality. The main difficulty is dealing with non-degenerate representations, for which \autoref{thm:ProjEssSubsp}
is essential.

\begin{thm}\label{thm:C*LpComm}(\cite{GarThi_ExtBid})
Let $A$ be a $C^*$-algebra. Then the following are equivalent:
\be\item $A$ can be isometrically represented on an $L^p$-space, for some $p\in [1,\I)\setminus\{2\}$;
\item $A$ can be isometrically represented on an $L^p$-space, for all $p\in [1,\I)$;
\item $A$ is commutative (and hence $A\cong C_0(X)$ for $X=\Max(A)$).\ee
\end{thm}

We also need to recall the following, whose proof can be found, for example, in~\cite{GarThi_GpsLpFunct}.
For a Banach algebra $A$, we denote its left (respectively, right, two-sided) multiplier algebra by
$M_L(A)$ (respectively, $M_R(A)$ and $M(A)$), and we write $\iota_L\colon A\to M_L(A)$
(respectively, $\iota_R\colon A\to M_R(A)$ and $\iota\colon A\to M(A)$) for the canonical
inclusion. When confusion may arise, we write the product in $M(A)$ with a dot.

\begin{thm}\label{thm:multipliers}
Let $A$ be a Banach algebra with a left (respectively, right, two-sided)
contractive approximate identity.
Let $X$ be a Banach space and let $\varphi\colon A\to \B(X)$ be a non-degenerate contractive
representation. Then there exists a unique unital contractive homomorphism $\psi_L\colon M_L(A)\to \B(X)$
(respectively, $\psi_R\colon M_R(A)\to \B(X)$ and $\psi\colon M(A)\to \B(X)$)
satisfying $\psi_L\circ\iota_L=\varphi$ (respectively, $\psi_R\circ\iota_R=\varphi$ and
$\psi\circ\iota=\varphi$).
Moreover, the map $\psi_L$ is given by
\[\psi_L(m)(\varphi(a)(\xi))=\varphi(m\cdot a)(\xi)\]
for all $m\in M_L(A)$, for all $a\in A$ and for all $\xi\in X$.
\end{thm}

The following theorem generalizes a result of Neufang and Runde, Theorem~2.2 in~\cite{NeuRun_column}, where the
authors assumed that the group in question was amenable and had a closed infinite abelian subgroup. Observe
that for $G=\Z$, \autoref{thm:FplambdaG} follows from the main result of \cite{GarThi_Quot}, since there it
is shown that for $p\neq 2$, there exists a semisimple quotient $A_V$ of $F^p(\Z)$ that is not
representable on $L^q$ for any $1\leq q<\I$. Indeed, if $F^p(\Z)$ were an operator algebra, so would be
any of its quotients. Recall that a unital commutative, semisimple operator
algebra is isometrically isomorphic to a closed subalgebra of $C(X)$ for a compact Hausdorff space $X$ (that is,
it is a uniform algebra). Hence, if $F^p(\Z)$ were an operator algebra, then $A_V$ would be a uniform algebra, hence isometrically representable on an $L^q$-space
for any $q\in [1,\I)$.

\autoref{thm:FplambdaG} will be used in the proof of \autoref{thm:FplambdaGLq}, which gives a much more general
result. In the proof below, for a locally compact group $G$ and $p\in [1,\I)$, and to emphasize the role played by $G$,
we will denote by $\gamma_{p,2}^G\colon F^p(G)\to C^*(G)$ the map from \autoref{thm:gammapq}.

\begin{thm}
\label{thm:FplambdaG}
Let $G$ be a locally compact group and let $p\in [1,\I)$.
Then one of $F^p(G)$, $F^p_\lambda(G)$, $PM_p(G)$ or $CV_p(G)$ is an operator algebra if and only if
either $p=2$ or $G$ is the trivial group.
\end{thm}
\begin{proof}
The ``if" implication is obvious if $G$ is the trivial group, since the associated Banach algebras
are all $\C$, while the statement is clear if $p=2$.

For the ``only if" implication, it is clear that if either $PM_p(G)$ or $CV_p(G)$ is an operator algebra, then so is $F^p_\lambda(G)$,
since there are isometric inclusions
\[F^p_\lambda(G)\subseteq PM_p(G)\subseteq CV_p(G).\]

Assume now that $F^p(G)$ is an operator algebra and that $p\neq 2$.
There is a canonical identification of $F^p(G)$ with $C^*(G)$, in view of \autoref{prop:FpGReprLq} and \autoref{rem:p2}.
Let $\kappa_p\colon F^p(G)=C^*(G)\to F^p_\lambda(G)$ denote the canonical contractive homomorphism with dense range.
It is well-known that the quotient $C^*(G)/\ker(\kappa_p)$ is a $C^*$-algebra. The induced map
\[\widehat{\kappa_p}\colon C^*(G)/\ker(\kappa_p)\to F^p_\lambda(G)\]
is an injective, contractive homomorphism. Now, \autoref{thm:C*incompress}
shows that $\widehat{\kappa_p}$ is isometric. Since $\kappa_p$, and hence $\widehat{\kappa}_p$, has dense range, it follows that
$F^p_\lambda(G)$ is isometrically isomorphic to $C^*(G)/\ker(\kappa_p)$.
In particular, $F^p_\lambda(G)$ is a $C^*$-algebra, and hence an operator algebra itself.

It is therefore enough to show the statement assuming that $F^p_\lambda(G)$ is an operator algebra.
Let $\Hi$ be a Hilbert space and let
$\varphi\colon F^p_\lambda(G)\to \B(\Hi)$ be an isometric representation.

\textbf{Claim:} $F^p_\lambda(G)$ is a $C^*$-algebra.
Note that $A$ has a (two-sided) contractive approximate identity, since so does $L^1(G)$ and there is a contractive
homomorphism $\iota\colon L^1(G)\to F^p_\lambda(G)$ with dense range. Moreover, since
\[\overline{\{\varphi(a)\eta\colon a\in F^p_\lambda(G),\eta\in\Hi\}}\]
is itself a Hilbert space, we may assume that the representation $\varphi$ is non-degenerate.
It follows from \autoref{thm:multipliers} that the algebra $M(F^p_\lambda(G))\subseteq \B(F^p_\lambda(G))$
of multipliers on $F^p_\lambda(G)$, is unitally representable on $\Hi$.

By Corollary~~2.5 in~\cite{GarThi_Wendel}, there is a canonical isometric identification of the
multiplier algebra $M(F^p_\lambda(G))\subseteq \B(F^p_\lambda(G))$, with the Banach algebra
\[C(F^p_\lambda(G))=\{x\in \B(L^p(G))\colon xa,ax\in F^p_\lambda(G) \mbox{ for all } a\in F^p_\lambda(G)\}\subseteq \B(L^p(G))\]
of centralizers of $F^p_\lambda(G)$.
Denote by $\psi\colon C(F^p_\lambda(G))\to \B(\Hi)$ the resulting unital, isometric representation.

There is an obvious identification of $G$ with a subgroup of the invertible isometries of $C(F^p_\lambda(G))$,
given by letting a group element $g\in G$ act on $L^p(G)$ as the convolution operator with respect to the
point mass measure $\delta_g$. Now, for $g\in G$, set $u_g=\psi(\delta_{g})$. Then $u_g$ is an invertible isometry
on $\Hi$, that is, a unitary operator. Moreover, the map $u\colon G\to\U(\Hi)$, given by $g\mapsto u_g$,
is easily seen to be a strongly-continuous unitary representation of $G$ on $\Hi$. The integrated form
$\rho_u\colon L^1(G)\to \B(\Hi)$ of $u$ is therefore a contractive, non-degenerate $\ast$-homomorphism.
Whence the subalgebra $\rho_u(L^1(G))\subseteq\B(\Hi)$ is closed under the adjoint operation.
Moreover, it is clear that the following diagram commutes:
\beqa\xymatrix{
L^1(G)\ar[d]_-\iota \ar[dr]^-{\rho_u} & \\
F^p_\lambda(G)\ar[r]_-{\varphi}&\B(\Hi).}
\eeqa
We conclude that $\varphi(F^p_\lambda(G))$, which equals $\overline{\rho_u(L^1(G))}$,
is a closed $\ast$-subalgebra of $\B(\Hi)$, that is, a \ca. The claim is proved.

It follows from \autoref{thm:C*LpComm} that $F^p_\lambda(G)$ is commutative.
Thus $G$ is itself commutative, and in particular $F^q_\lambda(G)=F^q(G)$ for all $q\in [1,\I)$,
by Theorem~3.7 in \cite{GarThi_GpsLp}.

The map $\gamma^G_{p,2}\colon F^p_\lambda(G)\to C^*_\lambda(G)$ is an isometric
isomorphism by \autoref{prop:FpGReprLq}. The fact that $\gamma^G_{p,2}$ is surjective implies that $G$ is finite,
by \autoref{thm:gammapq}. Using that $\gamma^G_{p,2}$ is isometric, we will show that $G$ must be the
trivial group.

Using finiteness of $G$, let $g\in G$ be an element with maximum order. Set $n=\mbox{ord}(g)\geq 1$,
and let $j\colon \Z_n\hookrightarrow G$ be the group homomorphism determined by $j(1)=g$. By Proposition~2.3
in~\cite{GarThi_GpsLpFunct}, there are natural isometric embeddings
$j_p\colon F^p_\lambda(\Z_n)\hookrightarrow F^p_\lambda(G)$
and $j_2\colon C^*_\lambda(\Z_n)\hookrightarrow C^*_\lambda(G)$.
Naturality of the maps involved implies that the following diagram is commutative:
\beqa
\xymatrix{
F^p_\lambda(\Z_n)\ar[r]^-{j_p}\ar[d]_-{\gamma^{\Z_n}_{p,2}} & F^p_\lambda(G)\ar[d]^-{\gamma^{G}_{p,2}} \\
C^*_\lambda(\Z_n)\ar[r]_-{j_2} & C^*_\lambda(G).
}\eeqa
In particular, $\gamma_{p,2}^{\Z_n}\colon F^p_\lambda(\Z_n)\to C^*_\lambda(\Z_n)$ is an isometric isomorphism.
(This map is really just the identity on $\C^n$.)

Set $\omega=e^{\frac{2\pi i}{n}}\in S^1$.
Using that $p\neq 2$ together with Proposition~2.8 in~\cite{GarThi_BanAlg} (see also the comments above it),
we conclude that if $x\in F^p_\lambda(\Z_n)$ is
an invertible isometry, then there exist $\zeta\in S^1$ and $k\in \{0,\ldots,n-1\}$ such that, under the
algebraic identification of $F^p_\lambda(\Z_n)$ with $\C^n$, we have
\[x=\left(\zeta, \zeta\omega^k,\ldots,\zeta\omega^{k(n-1)}\right).\]
In particular, if $n>1$, then not every element in $(S^1)^n\subseteq \C^n$ has norm one in $F^p_\lambda(\Z_n)$.
Since this certainly is the case in $C^*_\lambda(\Z_n)$, we must have $n=1$. By the choice of $n$, we conclude
that $G$ must be the trivial group, and the proof of the theorem is finished.
\end{proof}

In contrast with \autoref{thm:FplambdaG}, we point out that some $L^p$-operator group
algebras are \emph{contractively and isomorphically}
representable on Hilbert spaces. For example, for any finite group $G$, abelian or not, and for any $p\in [1,\I)$,
the map $\gamma_{p,2}\colon F^p(G)\to C^*(G)$ from \autoref{thm:gammapq}, is a
contractive isomorphism.

For a locally compact group $G$, we review the definitions of the Banach algebras $M^p(G)$ and $M^p_\lambda(G)$ from \cite{GarThi_GpsLpFunct}.
Recall that $M^1(G)$ is the (unital) Banach algebra of finite complex Radon measures on $G$, and it can be identified with
the multiplier algebra of $L^1(G)$. In particular, observe that the left regular representation
$\lambda_p\colon L^1(G)\to \B(L^p(G))$ can be extended to a representation $\lambda_p\colon M^1(G)\to \B(L^p(G))$.

\begin{df}
Let $G$ be a locally compact group, and let $p\in [1,\I)$. Define $M^p(G)$ to be the completion of $M^1(G)$ in the norm given by
\[\|\mu\|_{M^p(G)}=\sup\left\{\|\pi(\mu)\|\colon \pi\colon M^1(G) \to \B(L^p(\nu)) \mbox{ is a contractive homomorphism}\right\}\]
for $\mu\in M^1(G)$.

The algebra $M^p_\lambda(G)$ is the completion of $M^1(G)$ in the norm
\[\|\mu\|_{M^p_\lambda(G)}=\|\lambda_p(\mu)\|_{\B(L^p(G))}\]
for $\mu\in M^1(G)$.
\end{df}

\begin{rem}\label{rem:InclMpG}
For $p\in [1,\I)$, one can show that there are natural (isometric) inclusions
\[F^p(G)\subseteq M^p(G)\subseteq M(F^p(G))\]
and
\[F^p_\lambda(G)\subseteq M^p_\lambda(G) \subseteq M(F^p_\lambda(G)).\]
Somewhat less easy is the existence of an inclusion $M^p_\lambda(G)\subseteq CV_p(G)$ for $p\in (1,\I)$; see \cite{GarThi_GpsLpFunct}.
\end{rem}

The following is a particular case of a result in \cite{GarThi_GpsLpFunct}.

\begin{prop}\label{prop:Funct}
Let $p\in [1,\I)$, let $G$ be a locally compact group, and let $H\subseteq G$ be an amenable subgroup. Then
the inclusion $H\hookrightarrow G$ induces canonical isometric unital homomorphisms
\[M^p_\lambda(H)\to M^p_\lambda(G) \ \ \mbox{ and } \ \ M^p(H)\to M^p(G).\]
\end{prop}

The way in which the above proposition will be used is through the following lemma. We state
it and prove it in greater generality than needed here for use elsewhere.

\begin{lma}\label{lma:SubgpH}
Let $p\in [1,\I)$ and let $G$ be a non-trivial locally compact group. Denote by $A^p(G)$ any of the following
algebras: $F^p(G)$, $M^p(G)$, $F^p_\lambda(G)$, $M^p_\lambda(G)$, $PM_p(G)$ or $CV_p(G)$ (the last
two in the case $p>1$). If $A^p(G)$ is isometrically and non-degenerately representable on a Banach
space $X$, then there exist a cyclic subgroup $H$ (finite or infinite) of $G$, and a unital isometric
representation $\psi\colon F^p_\lambda(H)\to \B(X)$.
\end{lma}
\begin{proof}
Choose a non-trivial element $g\in G$, and denote by $H$ the (not necessarily closed) cyclic subgroup of $G$
generated by $g$. Since either $F^p(G)$ or $F^p_\lambda(G)$ is isometrically a subalgebra of $A^p(G)$ (see \autoref{rem:InclMpG}),
we may assume, without loss of generality, that $A^p(G)$ is either $F^p_\lambda(G)$ or $F^p(G)$. Denote
by $B^p(G)$ either $M^p_\lambda(G)$ (if $A^p(G)=F^p_\lambda(G)$) or $M^p(G)$ (if $A^p(G)=F^p(G)$).
By \autoref{prop:Funct}, and since $H$ is amenable and discrete, the inclusion of $H\subseteq G$
induces a canonical isometric embedding $F^p_\lambda(H)\hookrightarrow B^p(G)$.

Let $X$ be a Banach space and let $\varphi\colon A^p(G)\to \B(X)$ be an isometric and non-degenerate representation.
Note that $A^p(G)$ has a contractive approximate
identity, since so does $L^1(G)$ and there is a contractive homomorphism $L^1(G)\to A^p(G)$ with dense image.
By \autoref{thm:multipliers}, $\varphi$ can be extended to an
isometric unital representation $\psi\colon M(A^p(G))\to \B(X)$. Since there are inclusions
\[F^p_\lambda(H)\subseteq B^p(G) \subseteq M(A^p(G)),\]
the restriction of $\psi$ to $F^p_\lambda(H)$ is the desired isometric unital representation of $F^p_\lambda(H)$ on $X$.
\end{proof}

The next theorem will be needed in the proof of
\autoref{thm:FplambdaGLq}. 
Recall that the algebra $PM_p(G)$ of $p$-pseudomeasures is the closure of $F^p_\lambda(G)$ in $\B(L^p(G))$
with respect to the weak$^*$ topology (also called the \emph{ultraweak} topology)
induced by the (canonical) identification of $\B(L^p(G))$ with the dual of $L^p(G)\widehat{\otimes} L^{p'}(G)$ given by the pairing
determined on simple tensors by
\[\langle a,\xi\otimes\eta\rangle_{\B(L^p(G)),L^p(G)\widehat{\otimes} L^{p'}(G)}=\langle a\xi,\eta\rangle_{L^p(G),L^{p'}(G)}\]
for all $a\in \B(L^p(G))$, all $\xi\in L^p(G)$ and all $\eta\in L^{p'}(G)$.

We denote by $G^{\mathrm{op}}$ the opposite group of $G$. With $\mu$ and $\nu$ as in \autoref{subs:Not}, the inversion map 
$\iota\colon (G,\mu)\to (G^{\mathrm{op}},\nu)$ is in fact a measure-preserving group isomorphism. In particular, the 
$p$-convolution algebras associated to $G$ are canonically isometrically isomorphic to those associated to $G^{\mathrm{op}}$.

\begin{thm}\label{thm:duality}
Let $G$ be a locally compact group, and let $p\in (1,\I)$. Then there are canonical isometric isomorphisms
\[F^p(G)\cong F^{p'}(G), \ M^p(G)\cong M^{p'}(G), \  F^p_\lambda(G)\cong F^{p'}_\lambda(G),\]
\[M^p_\lambda(G) \cong M^{p'}_\lambda(G), \ PM_p(G)\cong PM_{p'}(G) \ \mbox{ and } \ CV_p(G)\cong CV_{p'}(G). \]
\end{thm}
\begin{proof}
For $F^p(G)$, this is Proposition~2.18 in~\cite{GarThi_GpsLp} (here recalled as \autoref{prop:dualityFullRed}),
while the proof for $M^p(G)$ is analogous, using $M^1(G)$ instead of $L^1(G)$.

For the reduced versions, we will show that there are canonical isometric isomorphisms $F^p(G)\cong F^{p'}(G^{\mathrm{op}})$, 
$M^p_\lambda(G) \cong M^{p'}_\lambda(G^{\mathrm{op}})$, $\ PM_p(G)\cong PM_{p'}(G^{\mathrm{op}})$ and $CV_p(G)\cong CV_{p'}(G^{\mathrm{op}})$,
since, by the remarks preceding this lemma, this implies the result.

From now on, and to minimize confusion, we will write $(G,\mu)$ and $(G^{\mathrm{op}},\nu)$ instead of $G$ and $G^{\mathrm{op}}$, 
to state explicitly which is the left Haar measure in each case. 

Define an isometric anti-isomorphism $\phi\colon L^1(G,\mu)\to L^1(G^{\mathrm{op}},\nu)$ by $\phi(f)(s)=\Delta(s)f(s)$ for all 
$f\in L^1(G,\mu)$ and all $s\in G$. Observe that the assignment $\xi\mapsto \xi\circ\iota$ defines an isometric isomorphism 
$L^{p'}(G,\mu)\cong L^{p'}(G^{\mathrm{op}},\nu)$.
We denote by $\psi\colon \B(L^p(G,\mu))\to \B(L^{p'}(G^{\mathrm{op}},\nu))$ the map given by
\[\psi(a)(\xi)(s)=[a'(\xi\circ\iota)](s^{-1})\]
for all $a\in \B(L^p(G,\mu))$, for all $\xi\in L^{p'}(G^{\mathrm{op}},\mu)$ and for all $s\in G$. (In other words, $\psi(a)$ is the transpose
of $a$, once $L^{p'}(G,\mu)$ is identified with $L^{p'}(G^{\mathrm{op}},\nu)$ via the assignment $\xi\mapsto \xi\circ\iota$.) It is easy to check that
$\psi$ is an isometric anti-isomorphism. 

We claim that the following diagram commutes:
\beqa
\xymatrix{
L^1(G,\mu)\ar[rr]^-{\phi}\ar[d]_-{\lambda^G_p} && L^1(G^{\mathrm{op}},\nu)\ar[d]^-{\lambda^{G^{\mathrm{op}}}_{p'}}\\
\B(L^p(G,\mu)) \ar[rr]_-\psi && \B(L^{p'}(G^{\mathrm{op}},\nu)).}
\eeqa

Let $f\in L^1(G)$, let $\xi\in L^{p'}(G^{\mathrm{op}},\nu)$ and let $s\in G$ be given. Recall (see, for example, Proposition~3.5 in~\cite{GarThi_GpsLp})
that the transpose of $\lambda_p^G(f)$ is $\lambda_{p'}^G(\phi(f)\circ\iota)$. 
(The function $\phi(f)\circ\iota\in L^1(G,\mu)$ is denoted by $f^\sharp$ in~\cite{GarThi_GpsLp}.)
Using this at the second step, we get
\begin{align*}
\psi(\lambda^G_p(f))(\xi)(s)&=[\lambda_p^G(f)'(\xi\circ\iota)](s^{-1})\\
&= [(\phi(f)\circ\iota)\ast(\xi\circ\iota)](s^{-1})\\
&= \int_G \Delta(t^{-1})f(t^{-1})\xi(st)\ d\mu(t)\\
&=\int_G f(t)\xi(st^{-1})\ d\mu(t).\end{align*}
On the other hand,
\begin{align*}
\lambda^{G^{\mathrm{op}}}_{p'}(\phi(f))(\xi)(s)&=(\phi(f)\ast_{\mathrm{op}}\xi)(s)\\
&=\int_G \Delta(t)f(t)\xi(st^{-1})\ d\nu(t)\\
&=\int_Gf(t)\xi(st^{-1})\ d\mu(t),
\end{align*}
which proves the claim.

It follows that there is a canonical isometric isomorphism $F^p_\lambda(G)\cong F^{p'}_\lambda(G^{\mathrm{op}})$, implemented
by $\psi$. By taking double
commutants, we deduce that $\psi$ also implements an isomorphism between $CV_p(G)$ and $CV_{p'}(G)$.
Extending the above diagram to $M^1(G,\mu)$ and $M^1(G^{\mathrm{op}},\nu)$, we
conclude that there is also a canonical isometric isomorphism $M^p_\lambda(G)\cong M^{p'}_\lambda(G^{\mathrm{op}})$.

To show the result for the algebras of pseudomeasures, we need to identify the ultraweak topologies on $\B(L^p(G,\mu))$ and
$\B(L^{p'}(G^{\mathrm{op}},\nu))$. In order to do this, observe that the canonical isometric identification
\[L^p(G,\mu)\widehat{\otimes}L^{p'}(G,\mu)\cong L^{p'}(G,\mu)\widehat{\otimes}L^p(G,\mu)\] induces the isometric isomorphism
$\psi\colon \B(L^p(G,\mu))\to \B(L^{p'}(G,\mu))$. Commutativity of the above diagram then implies that $\phi$ extends to an
isometric isomorphism between the ultraweak closures of $L^1(G,\mu)$ and $L^1(G^{\mathrm{op}},\nu)$, that is, to an isometric isomorphism
$PM_p(G)\cong PM_{p'}(G^{\mathrm{op}})$.
\end{proof}

The folowing is one of the main results of this paper. It determines precisely when one of the $p$-convolution algebras considered in the
literature can be represented on an $L^q$-space, for some $q\in [1,\I)$. It can be interpreted as stating that the $L^p$- and $L^q$-representation theories
of a nontrivial group, are incomparable whenever they are not ``obviously'' equivalent. We point out that this represents a
significant generalization of Theorem~2.2 in~\cite{NeuRun_column}, where the authors only deal with the case $q=2$, and moreover assume
that $G$ is amenable and has a closed infinite abelian subgroup.

The proof of \autoref{thm:FplambdaGLq} uses machinery from a number of other works. The reader is referred to the third page in
the introduction for an overview of the proof.

\begin{thm}
\label{thm:FplambdaGLq}
Let $G$ be a locally compact group, and let $p,q\in [1,\I)$ with $q>1$.
Then one (or all) of $F^p(G)$, $M^p(G)$, $F^p_\lambda(G)$, $M^p_\lambda(G)$, $PM_p(G)$ or $CV_p(G)$ is isometrically
representable on an $L^q$-space if and only if one of the following holds:
\be\item $G$ is the trivial group;
\item we have $\left|\frac{1}{2}-\frac{1}{p}\right|=\left|\frac{1}{2}-\frac{1}{q}\right|$; or
\item $p=2$ and $G$ is abelian.\ee
\end{thm}
\begin{proof}
We begin with the ``if" implication. When $G$ is the trivial group, all the Banach algebras in the statement are $\C$,
which is clearly representable on an $L^q$-space for any $q\in [1,\I)$. On the other hand, the identity
$\left|\frac{1}{2}-\frac{1}{p}\right|=\left|\frac{1}{2}-\frac{1}{q}\right|$ is equivalent to $p$ and $q$ being
either equal or conjugate. The case $p=q$ is trivial. If $p$ and
$q$ are conjugate, then \autoref{thm:duality} shows that all of the algebras in the statement are (canonically) representable on $L^q(G)$.
Finally, if $p=2$ and $G$ is abelian, then all of the Banach algebras in the statement are commutative \ca s, and hence have the
form $C_0(X)$ for some locally compact Hausdorff space $X$. It is then an easy exercise to check that for any $q\in [1,\I)$ and
for any such space $X$, there exist an $L^q$-space and an isometric representation of $C_0(X)$ on it.

We turn to the ``only if" implication. Since there are isometric inclusions
\[F^p_\lambda(G)\subseteq M^p_\lambda(G)\subseteq CV_p(G), \ F^p_\lambda(G)\subseteq PM_p(G)\subseteq CV_p(G) \ \mbox{ and } \ F^p(G)\subseteq M^p(G)\]
(see \autoref{rem:InclMpG}), we may assume that either $F^p_\lambda(G)$ or $F^p(G)$ is representable on an $L^q$-space.
By \autoref{prop:dualityFullRed}, and without loss of generality, we can assume that $p,q\in [1,2]$. Suppose that $p=2$ and $p\neq q$. Then
$F^p_\lambda(G)=C^*_\lambda(G)$ and $F^p(G)=C^*(G)$ are \ca s. By \autoref{thm:C*LpComm}, these must be abelian
\ca s, whence $G$ itself must be abelian.

Since the arguments for $F^p_\lambda(G)$ and $F^p(G)$ are very similar, we will carry them out together until we
have to distinguish the two cases. We therefore denote by $A^p(G)$ either $F^p_\lambda(G)$ or $F^p(G)$.

Now suppose that $p\neq 2$. Then $q\neq 2$ by \autoref{thm:FplambdaG}.
Let $E$ be an $L^q$-space and let
$\varphi\colon A^p(G)\to \B(E)$ be an isometric representation. Note that $A^p(G)$ has a contractive approximate
identity, since so does $L^1(G)$ and there is a contractive homomorphism $L^1(G)\to A^p(G)$ with dense image.
By \autoref{thm:ProjEssSubsp}, there exists a contractive idempotent $e\in \B(E)$ such that $e(E)$ is the essential subspace $E_0$ of $\varphi$.
Then $E_0$ is an $L^q$-space by Theorem~6 in~\cite{Tzafriri}. Denote by $\varphi_0\colon A^p(G)\to \B(E_0)$ the restriction
of $\varphi$. Then $\varphi_0$ is a non-degenerate isometric representation. By \autoref{lma:SubgpH}, there exist a cyclic
subgroup $H$ of $G$, and an isometric unital representation $\psi\colon F^p_\lambda(H)\to \B(E_0)$.
We claim that $\psi(F^p_\lambda(H))$ is isometrically isomorphic to $F^q_\lambda(H)$. We divide the proof into two cases.

\vspace{0.3cm}

\underline{Case 1:} $H\cong \Z$. Observe that $\psi(F^p_\lambda(\Z))$ and $F^q_\lambda(\Z)$
are both Banach algebras generated by an invertible isometry of an $L^q$-space and its inverse. Since $q\neq 2$, by part~(2) of Corollary~5.21
in~\cite{GarThi_BanAlg}, $F^q_\lambda(\Z)$ is the unique, up to (isometric) isomorphism, Banach algebra generated
by an invertible isometry of an $L^q$-space and its inverse whose Gelfand transform is \emph{not} an isomorphism (isometric or not).
Since $p\neq 2$, the Gelfand transform of $F^p_\lambda(\Z)$ is not an isomorphism, so the same is true for $\psi(F^p_\lambda(\Z))$.
It thus follows that there is an isometric isomorphism $\psi(F^p_\lambda(H))\cong F^q_\lambda(H)$, as desired. This proves the first case.

\vspace{0.3cm}

\underline{Case 2:} $H\cong \Z_r$ for $r\geq 2$. By replacing $H$ with a subgroup, we may assume
that $r$ is prime. Denote by $v\in F^p_\lambda(\Z_r)$ the canonical invertible
isometry generating $F^p_\lambda(\Z_r)$, and set $w=\psi(v)$, which is an invertible isometry of $E_0$. Then
\[\spec(w)=\spec(v)=\{\zeta\in S^1\colon \zeta^r=1\}.\]
Thus, the only admissible spectral configurations $\sigma=(\sigma_n)_{n\in\N}$ for $w$ are
\bi
\item $\sigma_0=\spec(w)$ and $\sigma_n=\emptyset$ for $n\geq 1$; or
\item $\sigma_0=\spec(w)=\sigma_r$ and $\sigma_n=\emptyset$ for $n\geq 1$ with $n\neq r$.\ei

In the first case, there is an isometric isomorphism between $F^q(w,w^{-1})=\psi(F^p_\lambda(\Z_r))$ and $(\C^r,\|\cdot\|_\I)$.
This would then contradict part~(5) of Theorem~3.5 in~\cite{GarThi_BanAlg}, since $p\neq 2$. In the second case, by the definition of the norm $\|\cdot\|_{\sigma,q}$ on
$F^q(\sigma)\cong F^q(w,w^{-1})$, there is an isometric isomorphism $F^q_\lambda(\Z_r)\to \psi (F^p_\lambda(\Z_r))$. This
finishes the proof of the claim. For later use, we stress the fact that, in the case $H=\Z_r$,
the isomorphism $\psi(F^p_\lambda(H))\cong F^q_\lambda(H)$
is in fact canonical, in the sense that the spectral configurations of $v$ and $\psi(v)$ agree, and
thus the isomorphism is induced by the identity on $\ell^1(H)$.

We deduce from the claim that there is an isometric isomorphism $F^p_\lambda(H)\cong F^q_\lambda(H)$. For $H\cong \Z$, \autoref{thm: FpZ not isom}
implies that $p=q$. For $H\cong \Z_r$, we use the fact that the isomorphism can be chosen to be canonical to prove that $p$ must
equal $q$.

So suppose that $p\neq q$ and $F^p_\lambda(\Z_r)\cong F^q_\lambda(\Z_r)$ canonically. Without loss of generality, we may assume that
$1\leq p<q <2$.
We claim that the map $\gamma_{p,2}\colon F_\lambda^p(\Z_r)\to C_\lambda^*(\Z_r)$ from \autoref{thm:gammapq}
is an isometric isomorphism. In view of Theorem~3.7 in \cite{GarThi_GpsLp},
this is equivalent to showing that
\[\|\lambda_p(f)\|_{\B(\ell^p(\Z_r))}=\|\lambda_2(f)\|_{\B(\ell^2(\Z_r))}\]
for all $f\in \ell^1(\Z_r)$.

Let $f\in \ell^1(\Z_r)$. Then $\|\lambda_2(f)\|_{\B(\ell^2(\Z_r))}\leq\|\lambda_p(f)\|_{\B(\ell^p(\Z_r))}$ by
\autoref{thm:gammapq}. Let $\theta\in (0,1)$ satisfy $\frac{1}{q}=\frac{\theta}{p}+\frac{1-\theta}{2}$.
By the Riesz-Thorin interpolation theorem, we have
\[\|\lambda_q(f)\|_{\B(\ell^q(\Z_r))}\leq \|\lambda_2(f)\|^\theta_{\B(\ell^2(\Z_r))}\|\lambda_p(f)\|_{\B(\ell^p(\Z_r))}^{1-\theta}.\]
Since $\|\lambda_p(f)\|_{\B(\ell^p(\Z_r))}=\|\lambda_q(f)\|_{\B(\ell^q(\Z_r))}$, we conclude that
\[\|\lambda_p(f)\|_{\B(\ell^p(\Z_r))}^\theta\leq \|\lambda_2(f)\|_{\B(\ell^2(\Z_r))}^\theta,\]
and hence $\|\lambda_p(f)\|_{\B(\ell^p(\Z_r))}\leq \|\lambda_2(f)\|_{\B(\ell^2(\Z_r))}$, as desired.
This proves the claim.

Since $\gamma_{p,2}$ is an isometric isomorphism, $F^p(\Z_r)$ is an operator algebra.
The result now follows from \autoref{thm:FplambdaG}.
\end{proof}

\begin{cor}\label{cor:FpGFqG}
Let $G$ be a locally compact group, and let $p,q\in [1,\I)$. Then the following are equivalent:
\be\item There is an isometric isomorphism $F^p(G)\cong F^q(G)$;
\item There is an isometric isomorphism $F_\lambda^p(G)\cong F_\lambda^q(G)$;
\item When $p>1$, there is an isometric isomorphism $PM_p(G)\cong PM_q(G)$;
\item When $p>1$, there is an isometric isomorphism $CV_p(G)\cong CV_q(G)$;
\item $\left|\frac{1}{p}-\frac 12\right|=\left|\frac{1}{q}-\frac 12\right|$.\ee
\end{cor}

In connection with \cite{GarThi_Wendel}, we mention here that for the algebras of pseudofunctions, the result above can be
improved to moreover allow isomorphisms between algebras associated to different groups, as follows: for locally compact groups
$G$ and $H$, and for $p,q\in [1,\I)$ not both equal to 2, there is an isometric isomorphism
\[F^p_\lambda(G)\cong F^q_\lambda(H)\]
if and only if $\left|\frac{1}{p}-\frac 12\right|=\left|\frac{1}{q}-\frac 12\right|$ and $G$ is isomorphic to $H$. (For discrete groups, the same
result holds with $PM_p$ or $CV_p$ instead of $F^p_\lambda$.)

\vspace{0.3cm}

A variant of the techniques used to prove \autoref{thm:FplambdaGLq} can be used to rule out representability on certain $QSL^p$-spaces.
(Recall that a Banach space $E$ is a $QSL^p$-space if it is isometrically isomorphic to a subspace of a quotient of an
$L^p$-space. Group representations on $QSL^p$-spaces are studied, for example, in \cite{NeuRun_column}.)
We are not able to use spectral configurations as in \autoref{thm:FplambdaGLq}, since we do not have a description of the
Banach algebra generated by an invertible isometry of a $QSL^p$-space. Instead, we will use the fact (see
Theorem~3.7 in~\cite{GarThi_GpsLp}) that $F^p_\lambda(\Z_n)$ and $F^p_\lambda(\Z)$ are universal with respect to representations
on $QSL^p$-spaces, together with the maps $\gamma_{p,q}$, to obtain the result.

The theorem below is only stated for $p,q\in [1,2]$, while the other exponents can be handled using duality.

\begin{thm}\label{thm:FplambdaGQSLq}
Let $G$ be a nontrivial locally compact group, and let $p,q\in [1,2]$.
\be
\item
If $q\leq p$, then \emph{all} of $F^p(G)$, $M^p(G)$, $F^p_\lambda(G)$, $M^p_\lambda(G)$, $PM_p(G)$ and $CV_p(G)$ can be isometrically
represented on a $QSL^q$-space.
\item If $p<q$, then \emph{none} of $F^p(G)$, $M^p(G)$, $F^p_\lambda(G)$, $M^p_\lambda(G)$, $PM_p(G)$ and $CV_p(G)$ can be isometrically
represented on a $QSL^q$-space.\ee
\end{thm}
\begin{proof}
If $q\leq p$, then every $L^p$-space is isomorphic to a closed subspace of an $L^q$-space. In particular, every $L^p$-space is
a $QSL^q$-space, and the result is immediate.

Suppose that $p<q$. As in the proof of \autoref{thm:FplambdaGLq}, it is enough to show that $F^p(G)$ and $F^p_\lambda(G)$
are not representable on a $QSL^q$-space. For convenience, we will denote by $A^p(G)$ one of these algebras, and by $B^p(G)$
either $M^p(G)$ or $M^p_\lambda(G)$, depending on which one $A^p(G)$ is denoting.

Let $E$ be a $QSL^q$-space and let $\varphi\colon A^p(G)\to \B(E)$ be an isometric representation. Since a subspace of a
$QSL^q$-space is obviously a $QSL^q$-space, we can assume that $\varphi$ is non-degenerate. By \autoref{lma:SubgpH}, there are
a cyclic subgroup $H$ of $G$ and a unital, isometric representation $\psi\colon F^p_\lambda(H)\to \B(E)$.
Unlike in the proof of \autoref{thm:FplambdaGLq}, we
cannot really conclude that $\psi(F^p_\lambda(H))$ is isomorphic to $F^q_\lambda(H)$ directly. Instead, and using the notation
from \cite{GarThi_GpsLp}, for $f\in \ell^1(H)$, we have (explanations below):
\[\|f\|_{L^p}=\|\lambda_p(f)\|_{\B(L^p(H))}=\|\psi(f)\|_{\B(E)}\leq \|f\|_{QSL^q}=\|f\|_{L^q}\leq \|f\|_{L^p}.\]
The first step is Theorem~3.7 in~\cite{GarThi_GpsLp} for $F^p(H)$ (the group $H$ is amenable); the second one is the fact that
$\psi$ is isometric; the third one is the definition
of the norm $\|\cdot\|_{QSL^q}$; the fourth one is the fact (Theorem~3.7 in~\cite{GarThi_GpsLp}) that $F^p_{\mathrm{QS}}(H)=F^p(H)$;
and the last one is \autoref{thm:gammapq}.

It follows that the map $\gamma_{p,q}^H\colon F^p_\lambda(H)\to F^q_\lambda(H)$ is an isometric isomorphism. In particular, $F^p_\lambda(H)$
is representable on an $L^q$-space, which contradicts \autoref{thm:FplambdaGLq}. The contradiction implies
that $A^p(G)$ cannot be isometrically represented on a $QSL^q$-space, as desired.
\end{proof}

We close this section with a question.
\autoref{thm:FplambdaGLq} asserts that a group \ca\ can be represented on an $L^p$-space, for some $p\neq 2$, if and only if the group is commutative.
More generally, \autoref{thm:C*LpComm} shows that a \ca\ can be represented on an $L^p$-space, for some $p\neq 2$, if
and only if it is commutative. This may well be a particular case of a more general fact ruling out representability of certain
$L^p$-operator algebras on $L^q$-spaces for two different, nonconjugate, H\"older exponents $p$ and $q$. On the other hand, the examples
in \cite{BleLeM_quotients} show that there are operator algebras, which are not \ca s, that can be represented on
$L^p$-spaces for $p\neq 2$. Therefore, we suggest:

\begin{qst} \label{qst}
Let $A$ be a unital Banach algebra, and suppose that the set
\[\{v\in A\colon v \mbox{ is invertible and } \|v\|=\|v^{-1}\|=1\}\]
has dense linear span in $A$. (This guarantees that $A$ is a $C^*$-algebra if it is an operator algebra.)
Let $p,q\in [1,2]$ with $p\neq q$. Suppose that $A$ can be isometrically represented on
an $L^p$-space and on an $L^q$-space. Is $A$ commutative? Does it follow that $A\cong C(X)$ for
some compact Hausdorff space $X$?\end{qst}

\section{An application to crossed products}
Let $X$ be a locally compact Hausdorff space, let $G$ be a locally compact group, and let $\alpha\colon G\to \mathrm{Homeo}(X)$ be a
topological action. For $p\in [1,\I)$, N.~Christopher Phillips defined in \cite{Phi_CPLp} the \emph{full} and
the \emph{reduced $L^p$-operator crossed products}
$F^p(G,X,\alpha)$ and $F^p_\lambda(G,X,\alpha)$ of the topological dynamical system $(G,X,\alpha)$,
generalizing the well established constructions in $C^*$-algebras, which are the case $p=2$.

In Question~8.2 of~\cite{Phi_CPLp}, Phillips asked whether, for a topological dynamical system $(G,X,\alpha)$
and for distinct $p,q\in [1,\I)$, there are any non-zero continuous homomorphisms
\[F^p(G,X,\alpha)\to F^q(G,X,\alpha) \ \ \mbox{ or } \ \ F^p_\lambda(G,X,\alpha)\to F^q_\lambda(G,X,\alpha).\]
While \autoref{thm:gammapq} shows that there may in general exist such homomorphisms (even contractive ones with dense range), in
\autoref{cor:CrossProdsIsom} we show that there is an isometric isomorphism $F^p(G,X,\alpha)\to F^q(G,X,\alpha)$
or $F^p_\lambda(G,X,\alpha)\to F^q_\lambda(G^{\mathrm{op}},X,\alpha)$ if and only if
$\left|\frac 1p - \frac 12\right|=\left|\frac 1q - \frac 12\right|$.

For the convenience of the reader, and since our notation is somewhat different,
we recall below the definitions of the crossed products (for $L^p$-operator algebras other than $C_0(X)$).
Recall that for $p\in [1,\I)$, we say that a Banach algebra
$A$ is an \emph{$L^p$-operator algebra} if it can be isometrically represented on an $L^p$-space.
The group of isometric automorphisms of $A$ is denoted $\Aut(A)$, and is always endowed with the strong topology.

The object we define next is the analog of the group algebra $L^1(G)$. All integrals are taken with respect to a fixed left Haar measure.

\begin{df}
Fix $p\in [1,\I)$.
Let $\alpha\colon G\to \Aut(A)$ be an action of a locally compact group $G$ on a Banach algebra $A$.
Denote by $L^1(G,A,\alpha)$ the Banach algebra completion of the space of continuous compactly supported functions $G\to A$
with respect to the $L^1$-norm, with product given by twisted convolution, that is,
\[(a \ast b)(g)=\int_G a(h)\alpha_h(b(h^{-1}g))\ dh\]
for $a,b\in C_c(G,A,\alpha)\subseteq L^1(G,A,\alpha)$ and $g\in G$.
\end{df}

Next, we define (regular) covariant representations and the associated (reduced) crossed product.

\begin{df}\label{df:LpCP}
Adopt the notation from the previous definition.
A \emph{(contractive) covariant representation} of $(G,A,\alpha)$ on an $L^p$-space $E$ is a pair $(\pi,v)$ consisting of a
nondegenerate contractive homomorphism $\pi\colon A\to \B(E)$ and an isometric group representation $v\colon G\to \B(E)$, satisfying
the covariance condition
\[v_g\pi(a)v_{g^{-1}}=\pi(\alpha_g(a))\]
for all $g\in G$ and for all $a\in A$. Given such a covariant representation, its
\emph{integrated form} is the nondegenerate contractive homomorphism $\pi\rtimes v\colon L^1(G,A,\alpha)\to \B(E)$ given by
\[(\pi\rtimes v)(a)(\xi)=\int_G \pi(a(g))v_g(\xi)\ dg\]
for all $a\in L^1(G,A,\alpha)$ and for all $\xi\in E$.

Given a contractive nondegenerate representation $\pi_0\colon A\to \B(E_0)$ on an $L^p$-space $E_0$, its associated \emph{regular covariant representation} is the
pair $(\pi,\lambda_p^{E_0})$ on $L^p(G)\otimes_p E_0\cong L^p(G,E_0)$ given by
\[\pi(a)(\xi)(g)=\pi_0\left(\alpha_{g^{-1}}(a)\right)(\xi(g)) \ \ \mbox{ and } \ \ (\lambda_p^{E_0})_g(\xi)(h)=\xi(g^{-1}h)\]
for all $a\in A$, for all $\xi\in L^p(G,E_0)$, and for all $g,h\in G$.

Denote by $\mathrm{Rep}_p(G,A,\alpha)$ the class of all contractive covariant representations of $(G,A,\alpha)$ on $L^p$-spaces,
and by $\mathrm{RegRep}_p(G,A,\alpha)$ the subclass of $\mathrm{Rep}_p(G,A,\alpha)$ consisting of the regular covariant representations.
The \emph{full crossed product} $F^p(G,A,\alpha)$ and the \emph{reduced crossed product} $F^p_\lambda(G,A,\alpha)$ are
defined as the completions of $L^1(G,A,\alpha)$ in the following norms:
\[\|a\|_{F^p(G,A,\alpha)}=\sup\{\left\|(\pi\rtimes v)(a)\right\|\colon (\pi,v) \in \mathrm{Rep}_p(G,A,\alpha)\}\]
and
\[\|a\|_{F_\lambda^p(G,A,\alpha)}=\sup\left\{\left\|(\pi\rtimes \lambda_p^{E_0})(a)\right\|\colon (\pi,\lambda_p^{E_0}) \in \mathrm{RegRep}_p(G,A,\alpha)\right\},\]
for all $a\in L^1(G,A,\alpha)$.
\end{df}

\begin{rem} In the context of the above definition, if $A$ is an $L^p$-operator algebra, then so will be $F^p(G,A,\alpha)$ and $F^p_\lambda(G,A,\alpha)$.\end{rem}

$L^p$-operator crossed products generalize group $L^p$-operator algebras, since for the one point space
$X=\{\ast\}$ we have $F^p(G,\ast,\id)=F^p(G)$ and $F^p_\lambda(G,\ast,\id)=F^p_\lambda(G)$.

\begin{rem}
Most results in~\cite{Phi_CPLp} assume
that the algebra $A$ is \emph{separable} and that the group is \emph{second countable}, and conclude that a number of Banach algebras
can be represented on \emph{$\sigma$-finite} $L^p$-spaces; see also Remark~1.18 in~\cite{Phi_CPLp}. (The purpose of using $\sigma$-finite
measure spaces is to apply Lamperti's theorem \cite{Lam_IsomLp}.) However, the theory can be developed without these countability assumptions,
and in fact the arguments in \cite{Phi_CPLp} go through in general, except that one gets an $L^p$-operator algebra that is not necessarily
representable on a $\sigma$-finite $L^p$-space. This, in particular, applies to Remark~4.6 and Proposition~4.8 in~\cite{Phi_CPLp}.\end{rem}

We note that while \autoref{df:LpCP} assumes covariant representations to be nondegenerate,
this assumption is unnecessary whenever $A$ has a (left or right) contractive approximate
identity, as can be shown with essentially the same argument as in \autoref{prop:defagree}. In fact,
our first preparatory result guarantees the existence of an approximate identity for full and reduced crossed
products whenever the underlying algebra has one.

\begin{thm}\label{thm:CPcai}
Let $p\in [1,\I)$, let $A$ be an $L^p$-operator algebra, let $G$ be a locally compact group, and let
$\alpha\colon G\to\Aut(A)$ be an action. Consider the following statements.
\be
\item $A$ has a left (right, two-sided) contractive approximate identity;
\item $F^p(G,A,\alpha)$ has a left (right, two-sided) contractive approximate identity;
\item $F^p_\lambda(G,A,\alpha)$ has a left (right, two-sided) contractive approximate identity. \ee
Then (1) implies (2), and (2) implies (3).

Finally, if $G$ is discrete, then also (3) implies (1), so they are all equivalent in this case.
\end{thm}
\begin{proof}
We prove the result for left contractive approximate identities; the proof for right contractive approximate identities is analogous.

(1) implies (2). Let $(a_\mu)_{\mu\in \Lambda}$ be a contractive approximate identity for $A$.
Set
\[\U=\{U\subseteq G \mbox{ open, such that } e\in U \mbox{ and } \overline{U} \mbox{ is compact}\}.\]
Given $U\in \U$, let $f_U\colon G\to [0,\I)$ be a continuous positive function with support contained in $U$ satisfying $\int\limits_G f_U(g)dg=1$.
Define a partial order on $\U$ by setting $U\leq V$ if $V\subseteq U$, and give $\U\times\Lambda$ the partial order given
by $(U,\mu)\leq (V,\nu)$ if and only if  $U\leq V$ and $\mu\leq \nu$. Recall that there exists a canonical contractive homomorphism
$\iota\colon L^1(G,A,\alpha)\to F^p(G,A,\alpha)$ with dense range. Write $\|\cdot\|_1$ for the norm on both $L^1(G)$ and $L^1(G,A,\alpha)$,
and $\|\cdot\|$ for the norm on both $A$ and $F^p(G,A,\alpha)$.

For $(U,\mu)\in \U\times\Lambda$, set $x_{(U,\mu)}=f_Ua_\mu$ (pointwise product), which is an element in $C_c(G,A,\alpha)\subseteq L^1(G,A,\alpha)$. It is
clear that
\[\|\iota(x_{(U,\mu)})\|\leq \|f_Ua_\mu\|_1=\|f_U\|_1\|a_\mu\|\leq 1\]
for all $(U,\mu)\in \U\times\Lambda$. Thus, in order to show that $\left(\iota(x_{(U,\mu)})\right)_{(U,\mu)\in \U\times\Lambda}$ is
a left contractive approximate identity for $F^p(G,A,\alpha)$, it is enough to show that $\left(x_{(U,\mu)}\right)_{(U,\mu)\in \U\times\Lambda}$ is
a left contractive approximate identity for $L^1(G,A,\alpha)$. For this, it is enough to consider functions in $L^1(G)\cdot A\subseteq L^1(G,A,\alpha)$,
since these span a dense subalgebra.

Let $f\in L^1(G)$ and $a\in A$ be given. We claim that $\left\|x_{(U,\mu)}\ast (fa)-fa \right\|\to 0$ as $(U,\mu)\to \infty$.
If either $f$ or $a$ is zero, then so is their product and there is nothing to show.
Without loss of generality, we may assume that $\|f\|_1=\|a\|=1$. Let $\ep>0$. We make the following choices:
\be
\item Using continuity of $\alpha$, choose an open set $U_1\subseteq G$ containing the unit of $G$ such that $\|\alpha_g(a)-a\|<\ep/3$
for all $g\in U_1$. Since $G$ is locally compact, we may assume that $U_1$ has compact closure, so that $U_1\in \U$.
\item Using continuity of the left regular representation on $L^1(G)$, choose an open set $U_2\subseteq G$ containing the unit of $G$ such that
$\int\limits_G\left|f(g^{-1}h)-f(h)\right|dg<\ep/3$ for all $g\in U_2$. Again, we may assume that $U_2\in \U$.
\item Choose $\mu_0\in \Lambda$ such that $\|a_\mu a-a\|<\ep/3$ for all $\mu\geq \mu_0$.
\ee

Set $U_0=U_1\cap U_2\in \U$. Given $(U,\mu)\geq (U_0,\mu_0)$, we have
\begin{align*}
\left\|x_{(U,\mu)}\ast (fa)-fa \right\| &= \int_G\left\|\int_Gx_{(U,\mu)}(g)\alpha_g((fa)(g^{-1}h))dg-f(h)a\right\|dh\\
&= \int_G\left\|\int_Gf_U(g)a_\mu f(g^{-1}h)\alpha_g(a)dg-f(h)a\right\|dh\\
&= \int_G\left\|\int_Gf_U(g)\left(f(g^{-1}h)a_\mu\alpha_g(a)-f(h)a\right)dg\right\|dh\\
&\leq \int_G\int_Gf_U(g)\left\|f(g^{-1}h)a_\mu\alpha_g(a)-f(h)a\right\|dgdh\\
&\leq \int_G\int_Gf_U(g)\left\|f(g^{-1}h)a_\mu\alpha_g(a)-f(h)a_\mu \alpha_g(a)\right\|dgdh\\
& \ \ + \int_G\int_Gf_U(g)\left\|f(h)a_\mu \alpha_g(a)-f(h)a_\mu a\right\|dgdh\\
&\ \ + \int_G\int_Gf_U(g)\left\|f(h)a_\mu a-f(h) a\right\|dgdh\\
&\leq \|a_\mu\|\|\alpha_g(a)\| \int_G\int_U f_U(g)\left|f(g^{-1}h)-f(h)\right|dhdg\\
& \ \ + \|a_\mu\|\|f\|_1 \int_U f_U(g) \|\alpha_g(a)-a\|dg + \|f_U\|_1\|f\|_1\|a_\mu a-a\|\\
&\leq \ep/3+\ep/3 +\ep/3 =\ep.
\end{align*}
We conclude that $F^p(G,A,\alpha)$ has a left contractive approximate identity.

(2) implies (3). This is immediate since there is a contractive homomorphism $\kappa_p\colon F^p(G,A,\alpha)\to F^p_\lambda(G,A,\alpha)$
with dense range; see Lemma~3.13 in~\cite{Phi_CPLp}.

Assume now that $G$ is discrete, and let us prove that (3) implies (1). Denote by $E\colon F^p_\lambda(G,A,\alpha)\to A$ the
faithful conditional expectation constructed in Proposition~4.8 of~\cite{Phi_CPLp} (observe that separability is not necessary; see
comments above).
Let $(x_\mu)_{\mu\in \Lambda}$ be a contractive left approximate identity for $F^p_\lambda(G,A,\alpha)$, and set $y_\mu=E(x_\mu)\in A$ for $\mu\in \Lambda$.
Using Remark~4.6 in~\cite{Phi_CPLp} (again, separability is not needed there),
we identify $A$ with a subalgebra of $F^p_\lambda(G,A,\alpha)$. For $a\in A$, we have
\[\lim_{\mu\in \Lambda} \|y_\mu a -a\|=\lim_{\mu\in \Lambda} \|E(x_\mu) a-a\|=\lim_{\mu\in \Lambda} \|E(x_\mu a-a)\|
\leq \lim_{\mu\in \Lambda} \|x_\mu a-a\|=0.\]
It follows that $(y_\mu)_{\mu\in \Lambda}$ is an left approximate identity for $A$, and it is clear
that $\|y_\mu\|\leq 1$ for all $\mu\in \Lambda$.
\end{proof}

Even when $G$ is discrete, the crossed products $F^p(G,A,\alpha)$ and $F^p_\lambda(G,A,\alpha)$ from the
proposition above may not contain $F^p(G)$ and $F^p_\lambda(G)$ canonically; this happens only when $A$
is unital. In general, however, they are canonically subalgebras of the multiplier algebras of the crossed
products, as we show below. Recall
our convention that the product in multiplier algebras is written with a dot.

\begin{thm} \label{thm:MultAlgCP}
Let $p\in [1,\I)$, let $A$ be an $L^p$-operator algebra with a left
contractive approximate identity, let $G$ be a locally compact group, and let
$\alpha\colon G\to\Aut(A)$ be an action by isometric isomorphisms. For $g\in G$, define a linear map
$L_g\colon L^1(G,A,\alpha)\to L^1(G,A,\alpha)$ by
\[L_g(a)(h) = \alpha_g(a(g^{-1}h))\]
for all $a\in L^1(G,A,\alpha)$ and all $h\in G$. Then the assignment $g\mapsto L_g$ induces natural contractive
homomorphisms
\[\iota^G_L\colon F^p(G)\to M_L(F^p(G,A,\alpha)) \ \mbox{ and } \ \iota^G_{L,\lambda}\colon F^p_\lambda(G)\to M_L(F^p_\lambda(G,A,\alpha)).\]

When $A$ has a right contractive approximate identity, for $g\in G$, the linear maps $R_g\colon L^1(G,A,\alpha)\to L^1(G,A,\alpha)$,
given by $R_g(a)(h)=a(hg)$ for all $a\in L^1(G,A,\alpha)$ and all $h\in H$, define natural contractive homomoprhisms
\[\iota^G_{R}\colon F^p(G)\to M_R(F^p(G,A,\alpha)) \ \mbox{ and } \ \iota^G_{R,\lambda}\colon F^p_\lambda(G)\to M_R(F^p_\lambda(G,A,\alpha)).\]

Finally, when $A$ has a two-sided contractive approximate identity, then the above maps define natural contractive homomorphisms
\[\iota^G\colon F^p(G)\to M(F^p(G,A,\alpha)) \ \mbox{ and } \ \iota^G_{\lambda}\colon F^p_\lambda(G)\to M(F^p_\lambda(G,A,\alpha)).\]

Moreover, the maps $\iota^G_{L,\lambda}, \iota^G_{R,\lambda}$ and $\iota^G_{\lambda}$ are isometric.
\end{thm}
\begin{proof}
We prove the theorem only for left approximate identities and left
multiplier algebras, but analogous proofs apply to the right and two-sided versions.

We claim that $L_g$ is a left multiplier on $L^1(G,A,\alpha)$, that is, $L_g(a\ast b)=L_g(a)\ast b$
for all $a,b\in L^1(G,A,\alpha)$. Given $a,b\in L^1(G,A,\alpha)$ and
$h\in G$, we have
\begin{align*}
L_g(a \ast b)(h)&=\alpha_g\left(\int_G a(k)\alpha_k\left(b(k^{-1}g^{-1}h)\right)dk\right)\\
&=\int_G \alpha_g(a(k))\alpha_{gk}\left(b((gk)^{-1}h)\right)dk\\
&=\int_G \alpha_g(a(g^{-1}t))\alpha_{t}\left(b(t^{-1}h)\right)dt\\
&= (L_g(a)\ast b)(h),
\end{align*}
as desired. It is clear that $g\mapsto L_g$ defines a strongly continuous group homomorphism $L\colon G\to M_L(L^1(G,A,\alpha))$.

\textbf{Claim 1:} The homomorphism $L$ extends to strongly continuous representations, $\iota^G$ and $\iota^G_\lambda$, of $G$ on
$F^p(G,A,\alpha)$ and $F^p_\lambda(G,A,\alpha)$ by isometric left multipliers.

Let $g\in G$ be given. To show that $L_g$ extends to isometric automorphisms $\iota^G_g$ and $(\iota^G_\lambda)_g$ of $F^p(G,A,\alpha)$
and $F^p_\lambda(G,A,\alpha)$, it is enough to show that for every covariant representation $(\pi,v)$ of $(G,A,\alpha)$ on an $L^p$-space,
one has
\[\|(\pi\rtimes v)(a)\|=\|(\pi\rtimes v)(L_g(a))\|\]
for all $a\in L^1(G,A,\alpha)$. To prove this, let $(\pi,v)$ be a covariant representation on an $L^p$-space $E$, and let $a\in L^1(G,A,\alpha)$. For $\xi\in E$,
we use the covariance identity $\pi(\alpha_g(x))=v_g\pi(x)v_{g^{-1}}$ at the third step, to get
\begin{align*}
(\pi\rtimes v)(L_g(a))(\xi)&=\int_G \pi(L_g(a)(h))v_h(\xi) \ dh\\
&= \int_G \pi(\alpha_g(a(g^{-1}h)))v_h(\xi) \ dh\\
&= \int_G v_g\pi(a(g^{-1}h))v_{g^{-1}h}(\xi) \ dh\\
&= v_g (\pi\rtimes v)(a)(\xi).
\end{align*}
Since $v_g$ is an isometry, it follows that $\|(\pi\rtimes v)(a)\|=\|(\pi\rtimes v)(L_g(a))\|$, as desired.

The fact that $\varphi_g$ and $(\iota^G_{L,\lambda})_g$ are left multipliers follows immediately from the fact that $L_g$ is a left multiplier, using
an $\ep/3$ argument. Finally, it is routine to check that the assignments $g\mapsto (\iota^G_{L})_g$ and
$g\mapsto (\iota^G_{L,\lambda})_g$ are strongly continuous actions of $G$. The claim is proved.

Since $L^1(G)$ is universal with respect to strongly continuous isometric actions of $G$,
there are contractive homomorphisms $\iota^G_L\colon L^1(G)\to M_L(F^p(G,A,\alpha))$ and
$\iota^G_{L,\lambda}\colon L^1(G)\to M_L(F_\lambda^p(G,A,\alpha))$, which admit explicit descriptions as follows. For $f\in L^1(G)$ and
$a\in L^1(G,A,\alpha)$, we have
\[(\iota^G_L(f)\cdot a)(g)=\int_Gf(h)L_h(a)(g)\ dh=\int_Gf(h)\alpha_h(a(h^{-1}g))\ dh,\]
for all $g\in G$. On the other hand, let $\kappa\colon F^p(G,A,\alpha)\to F_\lambda^p(G,A,\alpha)$ denote the canonical contractive
homomorphism with dense range. Since $\kappa$ is the identity on $L^1(G,A,\alpha)$, by \autoref{thm:CPcai} it maps a contractive left approximate identity
of $F^p(G,A,\alpha)$ to a contractive left approximate identity of $F^p_\lambda(G,A,\alpha)$. Thus there exists a unique unital extension
$\widetilde{\kappa}\colon M_L(F^p(G,A,\alpha))\to M_L(F^p_\lambda(G,A,\alpha)),$
and we have $\iota^G_{L,\lambda}=\widetilde{\kappa} \circ \iota^G_L$.

It remains to show that $\iota^G_{L,\lambda}$ is isometric when $L^1(G)$ is endowed with the
norm of $F^p_\lambda(G)$. We begin with some general observations.
Since $F^p(G,A,\alpha)$ and $F^p_\lambda(G,A,\alpha)$ have left contractive approximate identities by \autoref{thm:CPcai},
it follows from \autoref{thm:multipliers} that any contractive, nondegenerate representation of any these algebras extends to a contractive,
unital representation of its left multiplier algebra.
Let $(\pi,v)$ be a covariant representation of $(G,A,\alpha)$ on an $L^p$-space $E$. By applying \autoref{thm:ProjEssSubsp}, we can
assume that $\pi$ (and hence $\pi\rtimes v$) is nondegenerate. We write
$\widetilde{\pi\rtimes v}\colon M_L(F^p(G,A,\alpha))\to \B(E)$
for the extension of $\pi\rtimes v$ to the left multiplier algebra, and similarly with
$\widetilde{\pi\rtimes v}\colon M_L(F^p_\lambda(G,A,\alpha))\to \B(E)$
if $(\pi,v)$ is a regular covariant representation. By a slight abuse of notation, we denote also by $v\colon L^1(G)\to \B(E)$ the
integrated form of $v$, which is given by $v(f)\xi=\int\limits_G f(g)v_g(\xi) dg$ for $f\in L^1(G)$ and $\xi\in E$.

\textbf{Claim 2:} We have $\widetilde{\pi\rtimes v}(\iota_L^G(f))= v(f)$ for all $f\in L^1(G)$.

\noindent Let $f\in L^1(G)$, let $a\in L^1(G,A,\alpha)$, and let $\xi\in E$. Then
\begin{align*}
(\widetilde{\pi\rtimes v})(\iota_L^G(f))\left[(\pi\rtimes v)(a)\xi\right]&=(\pi\rtimes v)[\iota_L^G(f)\cdot a](\xi)\\
&=\int_G \pi\left[(\iota_L^G(f)\cdot a)(g)\right]v_g(\xi) \ dg\\
&=\int_G \pi\left[\int_G f(h)\alpha_h(a(h^{-1}g))\ dh\right]v_g(\xi)\ dg\\
&=\int_G f(h)\int_G\pi\left[\alpha_h(a(h^{-1}g))\right]v_g(\xi)\ dgdh\\
&=\int_G f(h)\int_G\pi\left[\alpha_h(a(k))\right]v_{hk}(\xi)\ dkdh\\
&=\int_G f(h)\int_Gv_h \pi\left[a(k)\right]v_{k}(\xi)\ dkdh\\
&=\int_G f(h)v_h\left((\pi\rtimes v)(a)\xi\right) dh\\
&=v(f)((\pi\rtimes v)(a)\xi),
\end{align*}
and the claim is proved.

Now suppose that $(\pi,v)$ is a regular covariant representation, so that there exists an $L^p$-space $E_0$ such that $v=\lambda^{E_0}_p$; see \autoref{df:LpCP}.
Let $f\in L^1(G)$, and observe that
\[\|\lambda_p^{E_0}(f)\|=\|\lambda_p(f)\otimes\id_{E_0}\|=\|\lambda_p(f)\|.\]
Using Claim~2 at the second step, and the above identity at the third step, we get
\begin{align*}
\|\iota^G_{L,\lambda}(f)\|_{M(F^p_\lambda(G,A,\alpha))}&=\sup\{\|(\pi\rtimes \lambda_p^{E_0})(\iota_L^G(f))\| \colon (\pi,\lambda_p^{E_0})\in\mathrm{RegRep}_p(G,A,\alpha)\}\\
&= \sup\{\|\lambda_p^{E_0}(f)\|\colon (\pi,\lambda_p^{E_0})\in\mathrm{RegRep}_p(G,A,\alpha)\} \\
&=\|\lambda_p(f)\|\\
&= \|f\|_{F^p_\lambda(G)}.
\end{align*}
We conclude that $\iota^G_{L,\lambda}\colon F^p_\lambda(G)\to M_L(F_\lambda^p(G,A,\alpha))$ is isometric. This finishes the proof.
\end{proof}

When $p=1$, and regardless of whether $G$ is amenable or not, 
there is a canonical identification $F^1(G)=F^1_\lambda(G)$ (see Proposition~2.11 in~\cite{GarThi_GpsLp}), and hence the maps $\iota^G_L$, $\iota^G_R$ and $\iota^G$ from
\autoref{thm:MultAlgCP} are isometric (because they agree with the ones defined on $F^1_\lambda(G)$). However, when $p>1$, it is not in general
true that the maps defined on $F^p(G)$ are isometric, or even injective with closed range, as we explain in the following example.

\begin{eg}
Let $G$ be a discrete group, and let $\texttt{Lt}$ denote the action of $G$ on $c_0(G)$ by left translation. By Theorem~4.3 in~\cite{GarThi_GpsLp},
the canonical map $F^p(G,G,\texttt{Lt})\to F_\lambda^p(G,G,\texttt{Lt})$ is an isometric isomorphism, regardless of $G$. Denote by $\kappa_p\colon F^p(G)\to F^p_\lambda(G)$
the canonical contractive map with dense range, and recall (\autoref{thm:Amen}) that $\kappa_p$ is a (not necessarily isometric) isomorphism if and only if $G$ is amenable.
Naturality of the maps involved implies that the following diagram of unital homomorphisms commutes:
\begin{align*}
\xymatrix{
F^p(G)\ar[d]_{\kappa_p} \ar[rr]^-{\iota^{G}}&& M(F^p(G,G,\texttt{Lt}))\ar[d]^{\cong} \\
F^p_\lambda(G) \ar[rr]_-{\iota_\lambda^{G}} && M(F_\lambda^p(G,G,\texttt{Lt})),
}
\end{align*}
where the vertical map on the right is the canonical one. Now, if $\iota^{G}$ were injective and had closed range, that is, if it were an isomorphism onto its range, then
it would follow that $\kappa_p$ is an isomorphism, and hence that $G$ is amenable.

Since an identical reasoning applies to left or right multiplier algebras, we conclude that for any non-amenable group $G$, and for the action $\texttt{Lt}$, the maps
$\iota_L^G$, $\iota_R^G$ and $\iota^G$ are not isomorphisms onto their ranges. \end{eg}

Let $\alpha\colon G\to\Aut(A)$ be an action of a locally compact group $G$ on a Banach algebra $A$.
We denote by $A^{\mathrm{op}}$ the opposite Banach algebra, and write $\alpha^{\mathrm{op}}\colon G\to\Aut(A^{\mathrm{op}})$ for the action given by
$\alpha_g^{\mathrm{op}}=\alpha_g$ for all $g\in G$. When $A$ is abelian, then clearly $A=A^{\mathrm{op}}$ and $\alpha=\alpha^{\mathrm{op}}$.

We note that in $L^1(G,A^{\mathrm{op}},\alpha^{\mathrm{op}})$, convolution is performed using opposite
product on $A$, which we denote by $\cdot_{\mathrm{op}}$ to minimize confusion. Recall that $\Delta\colon G\to \R_+$ denotes the modular function
(see \autoref{subs:Not}).

\begin{prop}\label{prop:dualityCP}
Let $p\in (1,\infty)$, let $A$ be an $L^p$-operator algebra, let $G$ be a locally compact group, and let $\alpha\colon G\to \Aut(A)$ be a
continuous action. Then the map
\[\theta\colon L^1(G,A,\alpha)\to L^1(G,A^{\mathrm{op}},\alpha^{\mathrm{op}})\]
given by $\theta(f)(s)=\Delta(s^{-1})\alpha_{s}(f(s^{-1}))$, for all $f\in L^1(G,A,\alpha)$ and all $s\in G$, is an isometric anti-isomorphism, which moreover extends to
isometric anti-isomorphisms
\[F^p(G,A,\alpha)\cong F^{p'}(G,A^{\mathrm{op}},\alpha^{\mathrm{op}}) \ \mbox{ and } \  F^p_\lambda(G,X,\alpha)\cong F^{p'}_\lambda(G,A^{\mathrm{op}},\alpha^{\mathrm{op}}).\]
In particular, $F^p(G,A,\alpha)$ and $F^p_\lambda(G,A,\alpha)$ are anti-isometrically representable on $L^{p'}$-spaces.
\end{prop}
\begin{proof}
Be begin by showing that the map $\theta$ from the statement is an isometric anti-isomorphism.

That $\theta$ is isometric follows from the definition of $\Delta$. For $f,g\in L^1(G,A,\alpha)$ and $s\in G$, we have
\[\theta(f\ast g)(s)=\Delta(s^{-1})\alpha_{s}(f\ast g)(s^{-1})=\Delta(s^{-1})\alpha_s\left(\int_G f(t)\alpha_t(g(t^{-1}s^{-1}))\ dt\right).\]
On the other hand, in the next computation we set $s^{-1}t=k$ at the fourth step to get
\begin{align*}
\left(\theta(g)\ast \theta(f)\right)(s)&= \int_G\theta(g)(t)\cdot_{\mathrm{op}}\alpha_t^{\mathrm{op}}\left(\theta(f)(t^{-1}s)\right) dt\\
&= \int_G \alpha_{t}(\Delta(st^{-1})\alpha_{t^{-1}s}(f(s^{-1}t))\Delta(t^{-1})\alpha_{t}(g(t^{-1}))\ dt\\
&= \Delta(s^{-1})\alpha_{s}\left(\int_G f(s^{-1}t)\alpha_{s^{-1}t}(g(t^{-1}))\ dt\right)\\
&= \Delta(s^{-1})\alpha_{s}\left(\int_G f(k)\alpha_{k}(g(k^{-1}s^{-1}))\ dk\right),
\end{align*}
which proves the claim.

Denote by $\iota\colon G\to G$ the inversion map, which is anti-multiplicative.
Let $\pi$ and $u$ be representations of $A$ and $G$, respectively, on an $L^p$-space $E$.
By abuse of notation, we denote by $\pi'\colon A\to \B(E')$ and $u'\colon G\to\B(E')$ the
homomorphisms given by $\pi'(a)=\pi(a)'$ and $u'_g=(u_g)'$ for all $a\in A$ and $g\in G$.
Then $\pi'$ and $u'\circ\iota$ are representations of
$A^{\mathrm{op}}$ and $G$, respectively, on the $L^{p'}$-space $E'$ (and conversely, by reflexivity). Moreover, for
$g\in G$ and $a\in A$, the identity $u_g\pi(a)u_{g^{-1}}=\pi(\alpha_g(a))$ is equivalent to $u'_{g^{-1}}\pi'(a)u'_g=\pi'(\alpha_g(a))$.
It follows that $(\pi,u)$ is a covariant representation for $(G,A,\alpha)$ if and
only if $(\pi',u'\circ \iota)$ is a covariant representation for
$(G,A^{\mathrm{op}},\alpha^{\mathrm{op}})$. This shows that the assignment $(\pi,u)\mapsto (\pi',u'\circ\iota)$ induces a natural bijection between the classes
$\mathrm{Rep}_p(G,A,\alpha)$ and $\mathrm{Rep}_{p'}(G,A^{\mathrm{op}},\alpha^{\mathrm{op}})$.
By the definition of the norm on the full crossed product, we conclude that
\[\|f\|_{F^p(G,A,\alpha)}=\|\theta(f)\|_{F^{p'}(G^,A^{\mathrm{op}},\alpha^{\mathrm{op}})}.\]
for all $f\in L^1(G,A,\alpha)$. This proves the statement for full crossed products.

The case of reduced crossed products follows similarly: the above bijection restricts to a bijection between
$\mathrm{RegRep}_p(G,A,\alpha)$ and $\mathrm{RegRep}_{p'}(G,A^{\mathrm{op}},\alpha^{\mathrm{op}})$, since for an
$L^p$-space $E_0$, the transpose of the representation $\lambda_p^{E_0}$ is $\lambda_{p'}^{E'_0}\circ \iota$.
\end{proof}

Note that if a Banach algebra $B$ is isometrically isomorphic to its opposite, then any Banach 
algebra completion of $B$ is also isomorphic to its opposite. We will use this (trivial) observation in the next corollary, 
with $B=L^1(G,A,\alpha)$ and the completions being the full and reduced crossed products.



\begin{cor}\label{cor:dualityAb} Adopt the notation of \autoref{prop:dualityCP}, and suppose that $A$ is abelian.
Then $L^1(G,A,\alpha)$ is canonically isometrically isomorphic to $L^1(G,A,\alpha)^{\mathrm{op}}$, and moreover there 
are natural isometric isomorphisms
\[F^p(G,A,\alpha)\cong F^{p'}(G,A,\alpha) \ \ \mbox{ and } \ \
 F_\lambda^p(G,A,\alpha)\cong F_\lambda^{p'}(G,A,\alpha).\]
\end{cor}
\begin{proof} 
Observe that $A$ is also an $L^{p'}$-operator algebra, since it is abelian.
By the first part of \autoref{prop:dualityCP}, and since $A$ is abelian, the map $\theta$ is a natural isometric isomorphism
$L^1(G,A,\alpha)\cong L^1(G,A,\alpha)^{\mathrm{op}}$. Upon taking completions with respect to all covariant representations of $(G,A,\alpha)$,
we conclude that $F^p(G,A,\alpha)$ is isometrically isomorphic to its
opposite algebra. Composing this isomorphism with the isomorphism $F^p(G,A,\alpha)^{\mathrm{op}}\cong F^{p'}(G,A,\alpha)$ given by \autoref{prop:dualityCP},
we obtain the desired isometric isomorphism for full crossed products.

The case of reduced crossed products is identical: one completes with respect to regular covariant representations instead.
\end{proof}

The following is the main result of this section. When $X$ is the one point space, we recover \autoref{thm:FplambdaGLq}.
We point out that we do not know how to prove \autoref{thm:CrossProds} directly without first obtaining some form of
\autoref{thm:FplambdaGLq}, and that \autoref{cor:FpGFqG} is not strong enough to deduce \autoref{thm:CrossProds} from it. 

\begin{thm}\label{thm:CrossProds}
Let $X$ be a locally compact Hausdorff space, let $G$ be a nontrivial locally compact group, and let $\alpha\colon G\to \mathrm{Homeo}(X)$ be a
topological action. Given $p,q\in [1,\I)$ with $q>1$, the Banach algebras $F^p(G,X,\alpha)$ and $F^p_\lambda(G,X,\alpha)$
can be isometrically represented on an $L^q$-space if and only if one of the following holds:
\be
\item $\left|\frac 1p - \frac 12\right|=\left|\frac 1q - \frac 12\right|$; or
\item $G$ is abelian, $p=2$ and the action $\alpha$ is trivial.
\ee
\end{thm}
\begin{proof}
If $\left|\frac 1p - \frac 12\right|=\left|\frac 1q - \frac 12\right|$, then the conclusion follows from \autoref{cor:dualityAb}.
Suppose that $p=2$, that $G$ is abelian and that the action is trivial.
Denote by $\widehat{G}$ the dual group of $G$. Then one easily checks
that $F^2_\lambda(G,X,\alpha)$ and $F^2(G,X,\alpha)$
are both isometrically isomorphic to $C_0(X\times \widehat{G})$, so they are representable on an $L^q$-space, for any $q\in [1,\I)$. This proves the ``if" implication.

Let us show the converse. We treat the case of reduced crossed products first. Let $p,q\in [1,\I)$ with $q>1$, let $E$ be an $L^q$-space and
let $\varphi\colon F^p_\lambda(G,X,\alpha)\to \B(E)$ be an isomeric isomorphism.
Since $C_0(X)$ has a (two-sided) contractive approximate identity, so does $F_\lambda^p(G,X,\alpha)$ by \autoref{thm:CPcai}.
Hence, upon restricting to its essential subspace and using \autoref{thm:ProjEssSubsp},
we may assume that $\varphi$ is non-degenerate.
Let $\widetilde{\varphi}\colon M(F^p_\lambda(G,X,\alpha))\to \B(E)$ be the extension of $\varphi$ provided by \autoref{thm:multipliers}.
By \autoref{thm:MultAlgCP}, there exists a canonical isometric homomorphism
$\iota^G_\lambda\colon F^p_\lambda(G)\to M(F^p_\lambda(G,X,\alpha))$. We conclude that $\widetilde{\varphi}\circ\iota^G_\lambda$ is an isometric representation of
$F^p_\lambda(G)$ on an $L^q$-space.
By \autoref{thm:FplambdaGLq}, either $\left|\frac 1p - \frac 12\right|=\left|\frac 1q - \frac 12\right|$ or $p=2$ and $G$ is
abelian. Assuming the latter, $F^2(G,X,\alpha)$ is a \ca. Suppose, without loss of generality, that $q\neq 2$. Then $F^2(G,X,\alpha)$,
and thus its multiplier algebra $M(F^2(G,X,\alpha))$, must be abelian by \autoref{thm:C*LpComm}. For $g\in G$, let $u_g\in M(F^2(G,X,\alpha))$ denote the
canonical unitary implementing $\alpha_g$. For $a\in C_0(X)\subseteq M(F^2(G,X,\alpha))$, we have
\[\alpha_g(a)= u_gau_g^*=a.\]
It follows that $\alpha$ is trivial, as desired. This shows the statement for $F^p_\lambda(G,X,\alpha)$.

We prove the statement for full crossed products now. Let $p,q\in [1,\I)$ with $q>1$, let $F$ be an $L^q$-space and
let $\psi\colon F^p(G,X,\alpha)\to \B(F)$ be an isomeric isomorphism. As before, we may assume that $\psi$ is nondegenerate,
and we denote by $\widetilde{\psi}\colon M(F^p(G,X,\alpha))\to \B(F)$ its unital extension. Let $g\in G\setminus\{1\}$, and denote by $H\leq G$ the
(not necessarily closed) cyclic subgroup of $G$ generated by $g$. Then $H$ is amenable and there is a commutative diagram of contractive homomorphisms
(explanations follow below)
\begin{align*}
\xymatrix{
F^p(H)\ar[rr]^-{\iota^{H,G}} \ar[drr]_{\iota^{H,G}_\lambda}&& M(F^p(G))\ar[rr]^-{\widetilde{\iota^G}}&& M(F^p(G,X,\alpha))\ar[d]^{\widetilde{\kappa}}\\
&& M(F_\lambda^p(G))\ar[rr]_-{\widetilde{\iota_\lambda^G}}&& M(F_\lambda^p(G,X,\alpha)).
}
\end{align*}
In the diagram above, $\iota^{H,G}$ and $\iota^{H,G}_\lambda$ are the canonical isometric inclusions provided by \autoref{rem:InclMpG} and \autoref{prop:Funct}; $\widetilde{\iota^G}$
and $\widetilde{\iota_\lambda^G}$ are the canonical unital extensions of the maps constructed in \autoref{thm:MultAlgCP}; and $\widetilde{\kappa}$ is the unital extension of the canonical
map $\kappa$ from full to reduced crossed product. Since $\iota^G_\lambda$ is isometric by \autoref{thm:MultAlgCP}, so is $\widetilde{\iota^G_\lambda}$. By commutativity of the diagram,
$\widetilde{\kappa}\circ\widetilde{\iota^G}\circ\iota^{H,G}$ is an isometric representation of $F^p(H)$ on the $L^q$-space $F$. By \autoref{thm:FplambdaGLq}, we must have
either $\left|\frac 1p - \frac 12\right|=\left|\frac 1q - \frac 12\right|$, or $p=2$. Assume that $p=2$ and $q\neq 2$.
Then $F^p(G,X,\alpha)$ is a $C^*$-algebra, so it must be abelian by
\autoref{thm:C*LpComm}. Since $C^*_\lambda(G)$ embeds into the abelian $C^*$-algebra $M(F_\lambda^2(G,X,\alpha))$, the group $G$ itself must be abelian. Finally, the same argument
used before shows that $\alpha$ must be trivial. This finishes the proof.
\end{proof}

Finally, the following corollary asserts that the $L^p$-crossed products obtained from topological dynamical systems, for varying $p$,
are pairwise non-isometrically isomorphic, except for conjugate exponents.

\begin{cor}\label{cor:CrossProdsIsom}
Let $X$ be a locally compact Hausdorff space, let $G$ be a locally compact group, and let $\alpha\colon G\to \mathrm{Homeo}(X)$ be a
topological action. Given $p,q\in [1,\infty)$, the following conditions are equivalent:
\be
\item $F^p(G,X,\alpha)$ is isometrically isomorphic to $F^q(G,X,\alpha)$;
\item $F^p_\lambda(G,X,\alpha)$ is isometrically isomorphic to $F^q_\lambda(G,X,\alpha)$;
\item $\left|\frac 1p - \frac 12\right|=\left|\frac 1q - \frac 12\right|$.\ee
\end{cor}

\providecommand{\bysame}{\leavevmode\hbox to3em{\hrulefill}\thinspace}
\providecommand{\MR}{\relax\ifhmode\unskip\space\fi MR }
\providecommand{\MRhref}[2]{%
  \href{http://www.ams.org/mathscinet-getitem?mr=#1}{#2}
}
\providecommand{\href}[2]{#2}

\end{document}